\documentclass[10pt]{amsart}
\usepackage[a4paper,hmargin={2.5cm,2.5cm},vmargin={2.5cm,2.5cm}]{geometry}

\usepackage{amsmath, amssymb, amsthm}
\usepackage[T1]{fontenc}
\usepackage{lmodern}
\usepackage{mathrsfs}
\usepackage{csquotes}
\usepackage[T1]{fontenc}
\usepackage[english]{babel}
\usepackage[colorlinks=true, linkcolor=black, urlcolor=black, citecolor=black]{hyperref}
\usepackage[shortlabels]{enumitem}
\title[Least Energy Solutions for Cooperative and  Competitive Schr\"odinger Systems]{Least Energy Solutions for Cooperative and Competitive Schr\"odinger  Systems with Neumann Boundary Conditions}

\author{Simone Mauro}
\author{Delia Schiera}
\author{Hugo Tavares}
\date{\today}

\usepackage[leqno]{amsmath}
\usepackage{amssymb,amsthm}

\usepackage{tikz-cd}

\theoremstyle{plain}
\newtheorem{theorem}{Theorem}[section]
\newtheorem{lemma}[theorem]{Lemma}
\newtheorem{proposition}[theorem]{Proposition}

\theoremstyle{plain}

\theoremstyle{plain}
\newtheorem{remark}[theorem]{Remark}

\RequirePackage{dsfont}
\newcommand{\N}{\field{N}}

\newcommand{\R}{\field{R}}

\newcommand{\field}[1]{\mathbb{#1}}
\newcommand{\wto}{\rightharpoonup}%convergenza debole
\newcommand{\tildeNb}{{\mathcal N}_\beta}%tilde nehari

\numberwithin{equation}{section}

\tikzset{%
    symbol/.style={%
        draw=none,
        every to/.append style={%
            edge node={node [sloped, allow upside down, auto=false]{$#1$}}}
    }
}
\usepackage{enumitem}

\tikzset{shorten <>/.style={shorten >=#1,shorten <=#1}}

\usepackage[swapnames,suftesi]{frontespizio}

\begin{document}
\maketitle
\begin{abstract}
We study the following gradient elliptic system with Neumann boundary conditions
\begin{equation*}
-\Delta u + \lambda_1 u = u^3 + \beta uv^2, \  -\Delta v + \lambda_2 v = v^3 + \beta u^2 v \  \text{in } \Omega,\qquad 
\frac{\partial u}{\partial \nu} = \frac{\partial v}{\partial \nu} = 0 \ \text{on } \partial \Omega,
\end{equation*}
where \( \Omega \subset \mathbb{R}^N \) is a bounded \( C^2 \) domain with \( N \leq 4 \), and \( \nu \) denotes the outward unit normal on the boundary. 
We investigate the existence of non-constant least energy solutions in both the cooperative (\( \beta > 0 \)) and the competitive (\( \beta < 0 \)) regimes, considering both the definite  and the indefinite case, namely $\lambda_1,\lambda_2\in\R$. We emphasize that our analysis includes both the subcritical case \( N \leq 3 \) and the critical case \( N = 4 \).
 
 Depending on the values of $\beta,\lambda_1,\lambda_2$, the least energy solution is obtained either via a linking theorem, by minimizing over a suitable Nehari manifold, or by direct minimization on the set  of  all non-trivial weak solutions. Our results and techniques can be also adapted to cover some previously untreated cases for Dirichlet conditions.

\end{abstract}

\noindent \textbf{Keywords:}  Critical and subcritical nonlinearities, gradient elliptic systems, least energy solutions, mixed cooperation and competition, Neumann boundary conditions, semilinear equations.\\
\noindent \textbf{2020 MSC:}  35B33,  	35B38,  35J47, 35J50, 35J57, 35J61

\section{Introduction}\label{introduction} 
 In this paper, we study  the existence of non-constant least energy solutions of the following system 
\begin{equation}\tag{$\mathcal P_\beta$}\label{Pb}
\begin{cases}
-\Delta u+\lambda_1 u=u^3+\beta uv^2&\text{in $\Omega$},\\
-\Delta v+\lambda_2 v=v^3+\beta u^2v&\text{in $\Omega$},\\
\frac{\partial u}{\partial\nu}=\frac{\partial v}{\partial\nu}=0&\text{on $\partial\Omega$},
\end{cases}
\end{equation}
where $\Omega\subset\R^N$ ($N\le4$) is a bounded domain with $\Omega$ a $C^2$ domain, and $\nu$ is the outward unit normal on the boundary. The values of the parameters $\lambda_i\in\R$, $i=1,2$, and $\beta\in\R$, will play a crucial role in our analysis.
The main goal is to find weak solutions $(u,v)$ with minimal energy and $u,v$ non-constant. 

 The system has several physical applications; for instance, it arises in nonlinear optics and in the study of Bose--Einstein condensates \cite{esry1997hartree,malomed2005multidimensional,merle2003stability}. In fact, the complex-valued functions $\psi_1 = e^{i\lambda_1 t} u(x)$ and $\psi_2 = e^{i\lambda_2 t} v(x)$ represent solitary wave solutions (solitons) of the coupled Schr\"odinger system:
\begin{equation*}\tag{$\mathcal S_\beta$}\label{Sb}
\begin{cases}
    -i\frac{\partial \psi_1}{\partial t} - \Delta \psi_1 = |\psi_1|^2 \psi_1 + \beta |\psi_2|^2 \psi_1 & \text{in $\Omega \times (0,+\infty)$}, \\
    -i\frac{\partial \psi_2}{\partial t} - \Delta \psi_2 = |\psi_2|^2 \psi_2 + \beta |\psi_1|^2 \psi_2 & \text{in $\Omega \times (0,+\infty)$}.
\end{cases}
\end{equation*}
The parameter $\beta\in\R$ represents the interaction  between the two waves, with positive or negative values indicating attractive or repulsive interaction, respectively.

From a mathematical point of view,  \eqref{Pb} is a prototype of a variational semilinear elliptic system with weak interactions. The later statement means that the system admits semi-trivial solutions, i.e. $(u,0),(0,v)\not\equiv(0,0)$ with $u,v$  solutions of 
\begin{equation}
\tag{$\mathcal Q_{\lambda}$}
\label{scalar eq}
-\Delta z+\lambda z=z^3\text{ in $\Omega$},\qquad
\frac{\partial z}{\partial\nu}=0\text{ on $\partial\Omega$},
\end{equation}
for $\lambda=\lambda_1$ and $\lambda=\lambda_2$, respectively.
We call $(0,0)$ a trivial solution of \eqref{Pb}; 
solutions $(u, v)$ such that $u, v \not \equiv 0$ will be called \emph{fully non-trivial solutions}.

\subsection{State of the art}
The existing literature provides several results on the scalar equation with Dirichlet boundary conditions, whereas Neumann boundary conditions are much less studied. Throughout this paper,  $\{\mu_k^D(\Omega)\}_{k\geq 1}$ denotes the  eigenvalues of the Laplacian with Dirichlet boundary conditions, while  $\{\mu_k(\Omega)\}_k=\{\mu_k^N(\Omega)\}_{k\geq 1}$ denotes the Neumann Laplacian eigenvalues.

\paragraph{\textbf{The scalar equation with Dirichlet boundary conditions}}
If $N\le3$ and $\lambda>-\mu_1^D(\Omega)$, it is standard to prove  that there exists a weak solution of 
\begin{equation*}
-\Delta z+\lambda z=z^3\text{ in $\Omega$},\qquad
z=0\text{ on $\partial\Omega$},
\end{equation*}
using the Mountain Pass Theorem or the Nehari method, see for instance \cite[Sections 1.4 and 4.2]{willem2012minimax}. Moreover, if $\lambda \le -\mu_1^D(\Omega)$, one can prove the existence of a nontrivial weak solution by means of a linking theorem \cite[Section 2.4]{willem2012minimax}, or alternatively by using the generalized Nehari method \cite[Chapter 4]{szulkin2010method}.  
While the linking theorem ensures the existence of a critical point with suitable geometric properties, the generalized Nehari method allows one to directly characterize a solution with minimal energy among all nontrivial solutions.

In the case $N = 4$, 
 the existence of solutions was first studied in the celebrated paper by Brezis and Nirenberg \cite{brezisnirenberg}, where the authors proved the existence of a (least energy) positive solution for $\lambda \in (-\mu_1^D(\Omega), 0)$. We recall that, for $\lambda \ge 0$, no solution exists in star-shaped domains, due to Pohozaev's identity.
In \cite{ambrosetti1986note} and \cite{capozzi1985existence}, the existence of a sign-changing solution under the weaker assumption $\lambda < 0$, $-\lambda \ne \mu_k^D(\Omega)$ for any $k \ge 1$, was proved. Furthermore, we refer to \cite{SzulkinWethWillem} for the existence of a least energy solution, i.e., a solution with minimal energy. We conclude this paragraph by mentioning the recent paper \cite{sun2025brezis}, where the authors give a complete answer to a conjecture concerning the existence of solutions to the Brezis-Nirenberg problem in dimension 3. We refer to the introduction of this paper for a complete state of the art on the topic.

\paragraph{\textbf{The scalar equation with Neumann boundary conditions}}
We now turn to the case of Neumann boundary conditions, problem \eqref{scalar eq}, which exhibits significant differences. In particular, when $\lambda > 0$, constant solutions exist, given by $z \equiv \pm \sqrt{\lambda}$. In the subcritical case $N \le 3$, the existence of non-constant solutions was proved in \cite{lin1988large, lin2006diffusion} for sufficiently large $\lambda > 0$. On the other hand, when $\lambda \le 0$, constant solutions do not exist, and one can adapt the techniques used in the Dirichlet case to prove the existence of a ground state solution. 

In the critical case $N = 4$, we refer the reader to \cite{adimurthimancini1,adimurthi1990critical, wang1991neumann} for the existence of non-constant solutions when $\lambda > 0$ is sufficiently large, and to \cite{comte1990solutions, comte1991existence} for the case $\lambda \in \mathbb{R}$ (we point out that, for $\lambda\leq 0$, solutions are necessarily sign-changing).  
The main differences with respect to the critical Dirichlet case are that, in the Neumann problem, solutions exist for \emph{all} $\lambda\in \R$, and that the regularity assumption $\partial\Omega \in C^2$ is required in order to handle the loss of compactness. 
\paragraph{\textbf{The system with Dirichlet boundary conditions}}
The system \eqref{Pb} with conditions $u,v=0$  on $\partial \Omega$, in the case $\lambda_1, \lambda_2 > -\mu_1^D(\Omega)$, was initially studied for $\Omega = \mathbb{R}^N$, with $N \le 3$, in \cite{ambrosetti2007standing,maia2006positive,mandel2015minimal,sirakov2007least} (and most of the statements are true also if $\Omega$ is bounded). 
Moreover, \cite{chen2012positive} is one of the first works in the critical case $N=4$, assuming $\lambda_1, \lambda_2 \in (-\mu_1^D(\Omega), 0)$; therein, the authors proved the existence of least energy solutions for large and small $\beta > 0$, and also for negative $\beta$. They also showed nonexistence of solutions for $\lambda_1, \lambda_2 \ge 0$ in star-shaped domains, exploiting Pohozaev's identity. The techniques developed in that work can be applied as well to prove existence in the subcritical case $N \le 3$.

More recently, for large $\beta > 0$, in \cite{clapp2022solutions} the authors proved existence of solutions for all values of $\lambda_1, \lambda_2\in \R$ when $N \le 3$, and for all $\lambda_1, \lambda_2 < 0$ with $-\lambda_1, -\lambda_2 \ne \mu_k^D(\Omega)$, for every $k\in\N$, when $N = 4$.

The case $|\beta| < 1$ and arbitrary $\lambda_1, \lambda_2$ in the subcritical regime was addressed in \cite{xu2025least}.  
However, the case $\beta < -1$ with $\lambda_1, \lambda_2 \le -\mu_1^D(\Omega)$ remains open even in the subcritical case; the only known result is a multiplicity result proved in \cite{noris2010existence}. See Theorem \ref{Estensione Dirichlet} below for our contributions in this case.

 \paragraph{\textbf{The system with Neumann boundary conditions}}
In the Neumann case (problem \eqref{Pb}), which is the main topic of our paper, the only available result appears in \cite{kou2022existence}, where the existence of non-constant solutions was studied in the particular case $\beta = \frac{1}{2}$ and $\lambda_1, \lambda_2 > 0$ sufficiently large. However, we point out that, as consequence of our Proposition \ref{sharp beta*} below, the solution obtained in \cite{kou2022existence} is semi-trivial. Different related systems are also considered in \cite{byeon2017pattern, chabrowski2001neumann, liu2013neumann, wei2008existence, zhang2012multiple}. On the other hand, the competitive case has not been addressed at all.

 \subsection{Main results}
Solutions of \eqref{Pb} correspond to critical points of the energy functional $\mathcal J_\beta :H:=H^1(\Omega)\times H^1(\Omega)\to\R$,
\[
\mathcal J_\beta(u,v)=\frac12\int_\Omega(|\nabla u|^2+\lambda_1 u^2)+\frac12\int_\Omega(|\nabla v|^2+\lambda_2 v^2)-\frac14\int_\Omega(u^4+2\beta u^2v^2+v^4). 
\]
Due to the continuous embedding 
\[
H^1(\Omega)\hookrightarrow L^{4}(\Omega),
\]
the energy functional is well defined on $H$, for $N\le4$. Notice that the embedding is compact if $N\le3$. Correspondingly, if $N\le3$ we say that  the system is subcritical, whereas if $N=4$ it is critical.

As mentioned before, the main aim of this paper is to find, under different assumptions on $\lambda_1,\lambda_2$ and $\beta$, \emph{least energy solutions} and \emph{ground state solutions}, which we define respectively as nontrivial solutions achieving the \emph{least energy level}
\begin{equation}\label{db}
l_\beta:=\inf\left\{\mathcal J_\beta(u,v)\ :\ (u,v)\in H,\ u,v\not\equiv0,\   \mathcal J_\beta'(u,v)=0\right\}
\end{equation}
and \emph{ground state level}:
\begin{equation*}
\label{eb}
g_\beta:=\inf\left\{\mathcal J_\beta(u,v)\ :\ (u,v)\in H\setminus \{(0,0)\},\   \mathcal J_\beta'(u,v)=0\right\}.
\end{equation*}

We highlight that, in the single equation case, the corresponding notions coincide, whereas this is not necessarily the case for systems, due to the presence of semi-trivial solitions.

\smallskip
\paragraph{\textbf{Cooperative case $\beta>0$}}
We first show the following
\begin{theorem}\label{thm ground state}
   Let $N\le4$, $\lambda_1,\lambda_2\in\R$, and $\beta >0$. The level $g_\beta$ is achieved, i.e. there exists a ground state solution of \eqref{Pb}.
\end{theorem}

We prove this result by using a linking theorem by Rabinowitz \cite{rabinowitz78linking}, and then minimizing directly $\mathcal{J}_\beta$ in the set of solutions $(u,v)\not \equiv (0,0)$. The linking level is also useful in the critical case $N=4$, as it allows to verify a compactness condition. In doing so, we were inspired by \cite{clapp2022solutions} where, however, only the case $\beta>0$ large is considered, a different linking theorem is used, and further restrictions on $\lambda_1,\lambda_2$ are needed for $N\leq 3$ (as mentioned in the state of the art section).

However, ground state solutions may be semi-trivial or constant. Precisely, 
when $\lambda_1,\lambda_2>0$, following \cite{mandel2015minimal}, one has:
\begin{proposition}\label{sharp beta*}
Let $\lambda_1, \lambda_2, \beta >0$ and $N \le 4$. Let
\begin{equation*}\label{def beta*}
    \beta_*:=\inf_{\substack{(u,v)\in H\\ u\cdot v\not\equiv0}}\frac{(4L)^{-1}B((u,v),(u,v))^2-|u|_4^4-|v|_4^4}{2|uv|_2^2}\ge1,\qquad 
\end{equation*}
where $L:=\min\{L_{\lambda_1},L_{\lambda_2}\}$, $L_{\lambda_i}$ is the least energy level of \eqref{scalar eq} with $\lambda=\lambda_i$ for $i=1,2$, and
\[B((u,v),(u,v)):=\int_\Omega|\nabla u|^2+\lambda_1\int_\Omega u^2+\int_\Omega|\nabla v|^2+\lambda_2\int_\Omega v^2.\]
\begin{itemize}
    \item[$(i)$] Every ground state is semi-trivial for $\beta< \beta_*$.  
    \item[$(ii)$] Every ground state is a least energy  solution for $\beta > \beta_*$. 
\end{itemize}
\end{proposition}

Moreover, 
the system \eqref{Pb} may admit constant solutions $(c_1,c_2)$ with $c_1,c_2\ne0$, which are all explicitly given in Proposition \ref{constant solution} in Section \ref{preliminaries}. In particular, for $\beta>0$, there are no non-constant solutions in the following cases: 
\begin{align}
\label{condizione 1.4 costanti}
 &(\lambda_1,\lambda_2)\notin(0,+\infty)^2 \text{ and } \beta>0,\\
 \label{condizione 1.5 costanti}
 &\lambda_1,\lambda_2>0 \text{ and  }\beta\in\left(\min\left\{\frac{\lambda_1}{\lambda_2},\frac{\lambda_2}{\lambda_1}\right\},\max\left\{\frac{\lambda_1}{\lambda_2},\frac{\lambda_2}{\lambda_1}\right\}\right). \end{align}
We establish suitable conditions for $\lambda_1,\lambda_2,\beta$ such that there exists a non-constant least energy solution. Define the constants:
 \begin{align}\label{const thm 1.1}
 K:=\frac{2\pi^{\frac N2}}{N\Gamma(\frac N2)},\quad K_q:=\frac{2\pi^{\frac N2}}{\Gamma(\frac N2)}\int_0^1(1-r)^qr^{N-1}\,dr,\quad q\ge1,\quad M:=\frac{(K+K_2)^2}{4K_4},
 \end{align}
 where $\Gamma(\cdot)$ is the Gamma function.  
 \begin{theorem}\label{Theorem 1.1}
  Let $N\le4$, $\lambda_1,\lambda_2\in\R$ and $\beta>0$.  
    \begin{itemize}
   \item[$(i)$]   There exists $\underline \beta(\lambda_1,\lambda_2)\in (0,1]$ such that, for $\beta<\underline \beta(\lambda_1,\lambda_2)$, the problem admits a  least energy solution, which is not constant in case 
    \eqref{condizione 1.4 costanti}  or \eqref{condizione 1.5 costanti} hold, or if 
\begin{equation}\label{eq non costanti w.c.}
\lambda_1,\lambda_2>0\quad \text{ and } \quad  M\left(\lambda_1^{\frac{4-N}{2}}+\lambda_2^{\frac{4-N}{2}}\right)\le\frac{\lambda_1^2 + \lambda_2^2 - 2\beta\lambda_1\lambda_2}{4(1 - \beta^2)}|\Omega|  
\end{equation}
    is satisfied.
 
  \item[$(ii)$]   
Any ground state solution is a non-constant least energy solution if either: \eqref{condizione 1.4 costanti} holds and $\beta>1$ is large enough (depending on $\lambda_1,\lambda_2$), or if 
 $ \lambda_1, \lambda_2 >0$, $\beta\ge\max\left\{\frac{\lambda_1}{\lambda_2},\frac{\lambda_2}{\lambda_1}\right\} $ and
\begin{equation}\label{eq least energy non constant}
\frac{(2K+K_2)^2}{8K_4(1+\beta)}(\lambda_1+\lambda_2)^{\frac{4-N}{2}}< \min\left\{\frac{C_S^2\min\{1,\lambda_1^2,\lambda_2^2\}}{4}, \frac{\lambda_1^2 + \lambda_2^2 - 2\beta\lambda_1\lambda_2}{4(1 - \beta^2)}|\Omega|\right\},
\end{equation}
where $C_S$ is the best constant for the embedding $ H^1(\Omega) \hookrightarrow L^4(\Omega)$, see \eqref{Sobolev constant H1} below.
\end{itemize}
\end{theorem}

\begin{remark}
    For a deep analysis about which $\lambda_1,\lambda_2,\beta$ satisfy conditions in $(i)$-$(ii)$, see Remarks \ref{remark constants iii} and \ref{remark constants i} below. See also Remark \ref{rem:lambda_1=lambda_2} for the particular case $\lambda_1=\lambda_2$.
\end{remark}

Regarding the proof of Theorem \ref{Theorem 1.1}, for $\beta$ small, according to Proposition \ref{sharp beta*}, a ground state solution is semi-trivial; thus, in the proof of $(i)$ we use a Nehari-type set (that takes into account the fact that $\lambda_i$ might be negative) to find directly a least energy solution, in the spirit of \cite{xu2025least}. For the proof of $(ii)$, we use instead Theorem \ref{thm ground state}, proving with further energy estimates that ground states are fully nontrivial and non-constant.

\smallskip 
\paragraph{\textbf{Competitive case $\beta<0$}}
 When \(\beta < 0\), one can exploit the same technique of the weakly cooperative case ($\beta>0$ small). However, it is difficult to prove that the Nehari-type set is a $C^1$-manifold and a natural constraint for every $\lambda_1,\lambda_2\in\R$. We prove that this is true if either $\lambda_1,\lambda_2>0$, or if $\lambda_1>0,\lambda_2=0$ and $N\le3$. Our results are as follow.

\begin{theorem}\label{Theorem 1.2}
Let \( N \leq 4 \), \( \lambda_1, \lambda_2 > 0\), and assume $\beta <0$. Then the system \eqref{Pb} admits a least energy (positive) solution.
Furthermore, every least energy solution is non-constant if either $\beta\leq -1$, or 
\begin{equation}\label{eq:Th.1.6_cond}
\beta \in (-1, 0)\ \text{ and } M(\lambda_1^{\frac{4-N}{2}} + \lambda_2^{\frac{4-N}{2}} + 2|\beta|\lambda_1^{\frac{4-N}{4}} \lambda_2^{\frac{4-N}{4}} )
< 
\lambda_1^2 + \lambda_2^2 + 2|\beta| \lambda_1 \lambda_2.
\end{equation}
\end{theorem}

Observe that \eqref{eq:Th.1.6_cond} is satisfied whenever $\beta\in (-1,0)$ is fixed, and either $\lambda_1$ or $\lambda_2$ is large enough.

\begin{theorem}\label{main result 4}
Let $N\le3$, $\lambda_1=:\lambda>0$ and $\lambda_2=0$, and assume $\beta <0$. There exists a least energy solution $(u,v)$ for \eqref{Pb} for $|\beta|$ large enough, which is not constant. 
\end{theorem}

We also consider the case $\lambda_1=\lambda_2\le0$. Since we do not know if the Nehari set is  a $C^1$-manifold in this case, we find a least energy solution minimizing $\mathcal{J}_\beta$ directly on the set of fully nontrivial solutions (in the spirit of \cite{xu2025least}), after proving that this set is non empty.  It is in this last part that the equality between the $\lambda_i$'s is important.

\begin{theorem}\label{thm 5}
   Let $N\le4$ and $\lambda := \lambda_1 = \lambda_2 \le 0$, and assume that $\beta <0$ is such that  $|\beta|$ is sufficiently small.
   Then there exists a least energy solution which is non-constant, if either $|\lambda|$ is not an eigenvalue or $\lambda = 0$.

\end{theorem}

\begin{theorem}\label{thm 6}
    Let $N\le3$, $\lambda:=\lambda_1=\lambda_2\le0$ and assume that $\beta <0$ with  $|\beta|$ large enough. Then
    \begin{itemize}
        \item[$(i)$] there exist infinitely many non-constant solutions for every $\lambda\le0$,
        \item[$(ii)$] there exists a least energy solution, which is not constant, for $\lambda=0$.
    \end{itemize}
\end{theorem}

\subsection{Some extensions to the Dirichlet case}

Regarding the Dirichlet case, by adapting some of the techniques of this paper, we can extend what was previously known (recall the state of the art section) and have the following new result.  Observe that with Dirichlet boundary conditions, we do not have to exclude constant solutions.

\begin{theorem}\label{Estensione Dirichlet}
Let $\Omega\subset\R^N$ be a bounded domain. 
The system 
\begin{equation*}
-\Delta u + \lambda_1 u = u^3 + \beta uv^2, \  -\Delta v + \lambda_2 v = v^3 + \beta u^2 v \  \text{in } \Omega,\qquad 
u= v = 0 \ \text{on } \partial \Omega
\end{equation*}
 has a least energy solution in the following cases:
\begin{itemize}
\item[$(i)$] $N=4$, $\lambda_1,\lambda_2<0$ and $|\lambda_1|,|\lambda_2|$ is not an eigenvalue, $\beta>0$ is small enough;
\item[$(ii)$] $N\le3$, $\lambda_1\ge0$ and $\lambda_2=\mu_1^D(\Omega)$, $\beta<0$ with $|\beta|$ large enough;
\item[$(iii)$] $N=4$, $\lambda:=\lambda_1=\lambda_2<0$ and $|\lambda|$ is not an eigenvalue, $\beta<0$ with $|\beta|$ small enough;
\item[$(iv)$] $N\le3$, $\lambda:=\lambda_1=\lambda_2=\mu_1^D(\Omega)$, $\beta<0$ with $|\beta|$ large enough.
\end{itemize}
\end{theorem}

Indeed, the proof of Theorem \ref{Estensione Dirichlet}-$(i)$ would follow the same proof as Theorem \ref{Theorem 1.1}-$(i)$ (one should only replace the estimates of Lemma \ref{Estimates for C} below with the ones of \cite[Lemma 3.4]{SzulkinWethWillem},   asking as a consequence that $|\lambda_1|,|\lambda_2|$ is not an eigenvalue). Theorem \ref{Estensione Dirichlet}-$(ii)$  would follow as Theorem \ref{main result 4}, while $(iii)$ and $(iv)$ follow Theorems \ref{thm 5} and \ref{thm 6}, respectively. The first and third items  extend the work of \cite{xu2025least} to the critical case, while  the second and fourth items provide partial answers for the less explored case $\beta < -1,\ \lambda_1, \lambda_2 \le -\mu_1^D(\Omega)$.

\paragraph{\textbf{Structure of the paper}}
We treat the cooperative case in Section~3, and the competitive case with $\lambda_1, \lambda_2 > 0$ in Section~4. The case $\beta < -1$ with $\lambda_2 = 0$ and $\lambda_1 > 0$ is addressed in Section~5, while Section~6 is devoted to the competitive case with $\lambda_1 = \lambda_2 \le 0$. Section~2 is devoted to preliminary notions, including a brief review of the scalar equation.

\subsection{Notation}
For $ f \in L^p(\Omega) $, $ |f|_p $ represents the $ L^p $-norm of $ f$, while, for $ \lambda > 0 $ and $ f \in H^1(\Omega) $, we define the (equivalent) $ H^1 $-norm as $ \|f\|_\lambda^2 := |\nabla f|_2^2 + \lambda |f|_2^2 $. If $\lambda=1$ we write $\|\cdot\|_1=\|\cdot\|$.

The space $ \mathcal{D}^{1,2}(\mathbb{R}^N) $ is the completion of $ C_c^\infty(\R^N) $ with respect to the norm $ |\nabla u|_2 $. We also denote $H=H^1(\Omega)\times H^1(\Omega)$ and, for every $N\le4$, 
\begin{equation}\label{Sobolev constant H1}
C_S:=\inf_{u\in H^1(\Omega)\setminus\{0\}}\frac{\int_\Omega|\nabla u|^2+\int_\Omega u^2}{\left(\int_\Omega u^4\right)^{\frac12}},
\end{equation}
namely  the best constant of the embedding $H^1(\Omega)\hookrightarrow L^4(\Omega)$. In dimension $4$, we also define the best constant of the embedding $\mathcal D^{1,2}(\R^4)\hookrightarrow L^4(\R^4)$:
\begin{equation}\label{Sobolev constant}
S:=\inf_{u\in\mathcal D^{1,2}(\R^4)\setminus\{0\}}\frac{\int_\Omega|\nabla u|^2}{\left(\int_\Omega u^4\right)^{\frac12}}.\end{equation}

\section{Preliminaries}\label{preliminaries}
We establish some general properties of weak solutions to \eqref{Pb}.
\begin{proposition}\label{constant solution}
Let $c_1,c_2\in\R$. The couple $(c_1,c_2)$ solves \eqref{Pb} if and only if one of the following conditions hold 
\begin{align}
   (c_1,c_2)=(\pm\sqrt{\lambda_1},0),\quad \text{when $\lambda_1\ge0$}, \quad (c_1,c_2)=(0,\pm\sqrt{\lambda_2}),\quad \text{when $\lambda_2\ge0$},\nonumber \\
    c_1^2+c_2^2=\lambda_1,\quad \text{when $\lambda_1=\lambda_2>0$ and $\beta=1$},\quad c_1^2-c_2^2=\lambda_1,\quad \text{when $\lambda_1=-\lambda_2$ and $\beta=-1$},\nonumber \\
    \label{costanti non banali}
    (c_1,c_2)=\left(\pm\sqrt{\frac{\lambda_1-\beta\lambda_2}{1-\beta^2}}, \pm\sqrt{\frac{\lambda_2-\beta\lambda_1}{1-\beta^2}}\right)\quad \text{when $\frac{\lambda_1-\beta\lambda_2}{1-\beta^2},\frac{\lambda_2-\beta\lambda_1}{1-\beta^2}\ge0$. }
\end{align}
\end{proposition}
\begin{proof}
    The couple $(c_1,c_2)$ solves \eqref{Pb} if and only if
    \begin{equation}\label{system costants}
        \lambda_1 c_1=c_1^3+\beta c_1c_2^2,\qquad
        \lambda_2 c_2=c_2^3+\beta c_1^2c_2.
    \end{equation}
    Now, if $c_1=0$ and $c_2\ne0,$ we obtain $c_2=\pm\sqrt{\lambda_2}$ when $\lambda_2\ge0$. Similarly, if $c_2=0$, we get $c_1=\pm\sqrt{\lambda_1}$  for $\lambda_1\geq 0$. Hence, we have the trivial constant solution $(0,0)$ and the semi-trivial constant solutions $(\pm\sqrt{\lambda_1}, 0), (0,\pm\sqrt{\lambda_2})$, when they are well defined. Otherwise, we can simplify system \eqref{system costants} as follows:
\begin{equation}\label{costanti ammissibili}
\begin{cases}
    \lambda_1=c_1^2+\beta c_2^2\\
    \lambda_2=c_2^2+\beta c_1^2
\end{cases}
\iff
\begin{cases}
    c_1^2=\frac{\lambda_1-\beta\lambda_2}{1-\beta^2}\\
    c_2^2=\frac{\lambda_2-\beta\lambda_1}{1-\beta^2},
\end{cases}
\end{equation}
when they are defined. Observe that, when $\beta = 1$, the system \eqref{system costants} reduces to $\lambda_1 = \lambda_2 = c_1^2 + c_2^2$, whereas, when $\beta = -1$, we obtain $\lambda_1 = \lambda_2 = c_1^2 - c_2^2$. This concludes the proof.
\end{proof}

\begin{remark}
We note that, for every $(u, v)$ solution of \eqref{Pb}, $u$ is a constant if and only if $v$ is a constant. Then the couples $(c_1,v), (u,c_2)$ with $u,v$ non-constant functions and $c_1,c_2\in \R$ are not admissible solutions.
\end{remark}

\begin{proposition}\label{regularity}
    Let $(u,v)\in H$ be a weak solution of \eqref{Pb}. Then $u,v\in C^{2,\gamma}(\Omega)\cap C^{1,\gamma}(\overline\Omega)$. Furthermore, if $\partial\Omega\in C^{2,\alpha}$ for some $\alpha\in(0,1)$, then there exists $\tilde\gamma\in(0,1)$ such that $u,v\in C^{2,\tilde\gamma}(\overline\Omega)$.
\end{proposition}
\begin{proof}
    If $(u,v)$ is a weak solution for \eqref{Pb} we can write 
  \[
  -\Delta u=(u^2+\beta v^2-\lambda_1)u,\    -\Delta v=(v^2+\beta u^2-\lambda_2)v \text{ in } \Omega,\quad   \frac{\partial u}{\partial\nu}=\frac{\partial v}{\partial\nu}=0 \text{ on } \partial\Omega.
  \]
Let $a_i(x):=u^2(x)+\beta v^2(x)-\lambda_i\in L^{\frac N2}(\Omega)$ for $i=1,2$ . Using the Brezis-Kato estimates \cite{breziskato1} (based on the Moser iteration \cite{moser1960new}, see also \cite[Lemmas 2.1, 2.2]{saldana2022least} for the regularity theory adapted to the Neumann case), we obtain that $u,v\in W^{2,p}(\Omega)$ for every $p\ge1$. By Sobolev embeddings one obtains interior $C^{2,\gamma}-$regularity and $C^{1,\gamma}-$regularity up to the boundary for some $\gamma\in(0,1)$, since $\partial\Omega\in C^2$. If $\partial\Omega\in C^{2,\alpha}$ for some $\alpha\in(0,1)$,   we have that $u,v\in C^{2,\tilde\gamma}(\overline\Omega)$ with $\tilde\gamma\in(0,1)$, by the Schauder estimates \cite{nardi2015schauder}. 
\end{proof}

\begin{remark}
Let $\beta>1$ and $\lambda_1=\lambda_2=\lambda>0$. According to \cite[Theorem 4.2 and Remark-$(1)$ p. 7]{weiyao2012uniqueness}, any solution for :
\[
-\Delta u+\lambda u=u^3+\beta uv^2,\  -\Delta v+\lambda v=v^3+\beta u^2v\ \text{ in } \Omega,\quad
u,v>0\ \text{ in } \Omega,\quad \frac{\partial u}{\partial \nu}=\frac{\partial v}{\partial\nu}=0 \text{ on } \partial\Omega.
\]
is of the type $u=v$. However, the proof therein does not work for $\lambda<0$, and we leave as an open problem to check whether or not, also in this case, the components coincide.
\end{remark}

We recall an inequality that will be essential for analyzing the compactness issue in the critical case.

\begin{lemma}[Cherrier's inequality \cite{cherrier1984meilleures}]\label{Cherrier}
Let $\Omega\subset\R^4$ be a bounded domain with $\partial\Omega\in C^1$. For each $\varepsilon>0$ there exists $M_{\varepsilon}>0$ such that
\[|u|_{2^*}\le\left(\frac{\sqrt2}{S}+\varepsilon\right)^{1/2}|\nabla u|_2+M_\varepsilon|u|_2,\ \forall\   u\in H^1(\Omega).\]
\end{lemma}
Here, $S$  is the best constant of the embedding $\mathcal{D}^{1,2}(\R^4)\hookrightarrow L^{4}(\R^4)$, see \eqref{Sobolev constant}. A minimizer for $S$ (up to scaling) is a solution for $-\Delta w=w^3$ in $\R^4$ and the family of the solutions to this equation in $\mathcal D^{1,2}(\R^4)$ is given by
 \begin{equation*}\label{bubbles}
 U_{\varepsilon,x_0}(x)=\frac{(8\varepsilon)^{1/2}}{\varepsilon+|x-x_0|^2},\ \ \varepsilon>0,\ \ x_0\in\R^4.
 \end{equation*}
Furthermore, $\boldsymbol{V}_\varepsilon=\left( \frac{U_{\varepsilon,0}}{\sqrt{1+\beta}},\frac{U_{\varepsilon,0}}{\sqrt{1+\beta}} \right)$  are weak solutions in $\mathcal D^{1,2}(\R^4)$ of 
\[
\tag{${\mathcal P}_{\infty,\beta}$}\label{sistema R^4}
    -\Delta u=u^3+\beta uv^2,\   -\Delta v=v^3+\beta u^2v\ \text{ in } \R^4,\quad    u,v\in\mathcal D^{1,2}(\R^4).
\]
Additionally, from \cite[Theorem 1.5]{chen2012positive}, $\boldsymbol{V}_\varepsilon$ is a least energy solution of \eqref{sistema R^4},
and the energy level of the associated energy functional is $\frac{S^2}{2(1+\beta)}$. It achieves also
\begin{equation}\label{Sinfbeta}
\begin{aligned}
S_{\infty,\beta}:&=\inf_{\substack{u,v\in\mathcal D^{1,2}(\R^4)\\u,v\not\equiv0}}\frac{|\nabla u|_2^2+|\nabla v|_2^2}{(|u|_4^4+2\beta|uv|_2^2+|v|_4^4)^{1/2}}\\
&=\frac{\left|\nabla \left(\sqrt{\frac{1}{1+\beta}}U_{\varepsilon,0}\right)\right|_2^2+\left|\nabla \left(\sqrt{\frac{1}{1+\beta}}U_{\varepsilon,0}\right)\right|_2^2}{\left(\left|\sqrt{\frac{1}{1+\beta}}U_{\varepsilon,0}\right|_4^4+2\beta\left|\sqrt{\frac{1}{1+\beta}}U_{\varepsilon,0}\cdot \sqrt{\frac{1}{1+\beta}}U_{\varepsilon,0}\right|_2^2+\left|\sqrt{\frac{1}{1+\beta}}U_{\varepsilon,0}\right|_4^4\right)^{1/2}}
=\frac{\sqrt2 S}{\sqrt{1+\beta}}.
\end{aligned}
\end{equation}
Since $-\Delta$ is invariant under traslations and rotations, and $\partial\Omega\in C^2$, we will assume from now on, without loss of generality,  that 
\[
\Omega\subset\{x_4>0\}\quad \text{ and }  0\in\partial\Omega \text{ has strictly positive (mean) curvature. }
\]
This geometric property is essential to obtain estimates in the critical case $(N=4)$ for $w_\varepsilon:=\eta U_{\varepsilon,0}$, where $\eta$ is a smooth cut-off function such that $\eta\equiv1$ in the ball $B_{\frac\rho2}(0)$ and $\eta\equiv0$ outside the ball $B_\rho(0)$.

  \begin{lemma}\label{Estimates for C}
  Let $N=4$. We have, as $\varepsilon\to0^+$, that there exist $C_1,C_2>0$ with $C_1>C_2$ such that
     \begin{align}
&         |\nabla w_{\varepsilon}|_2^2=\frac{S^2}{2}- C_1\sqrt\varepsilon+O(\varepsilon|\log\varepsilon|),\qquad         |\nabla w_\varepsilon|_1=O(\sqrt\varepsilon), \nonumber \\ 
& |w_\varepsilon|_1=O(\sqrt\varepsilon),\qquad       |w_\varepsilon|_3^3=O(\sqrt\varepsilon),\qquad         |w_{\varepsilon}|_4^4=\frac{S^2}{2}- C_2\sqrt\varepsilon+O(\varepsilon)\nonumber \\
        &          |w_\varepsilon|_2^2= O(\varepsilon|\log\varepsilon|), \label{stima norma 2} 
     \end{align}
 for any $\lambda_i\in\R$ and for $\varepsilon>0$ small enough.
 \end{lemma}
 \begin{proof}
    The estimates of $|\nabla w_\varepsilon|_2^2, |w_\varepsilon|_2^2, |w_\varepsilon|_4^4$    are  proved  in \cite[Lemma 2.2, pages 13-14]{adimurthimancini1}.  The estimates of $|w_\varepsilon|_1,|w_\varepsilon|_3^3$ can be deduced by \cite[Lemma 2.25, eq 2.12]{willem2012minimax}, see also \cite[Lemma A.2]{pistoia2017spiked}.
       As for the estimate of $|\nabla w_\varepsilon|_1$:
    \begin{align*}
    \int_\Omega|\nabla w_\varepsilon|&\le\frac12\int_{B_\rho}\eta|\nabla U_{\varepsilon,0}|+\frac12\int_{B_\rho}|\nabla \eta|U_{\varepsilon,0}\le C\sqrt\varepsilon\int_{B_\rho}\frac{|x|}{(\varepsilon+|x|^2)^2}\,dx+C\sqrt\varepsilon\int_{B_\rho}\frac{1}{\varepsilon+|x|^2}\,dx\\
   &=C_1\varepsilon\int_0^{\frac{\rho}{\sqrt\varepsilon}}\frac{r^4}{(1+r^2)^2}+C_1\varepsilon^{\frac32}\int_0^{\frac{\rho}{\sqrt\varepsilon}}\frac{r^3}{1+r^2}\le C_2\sqrt\varepsilon .\qedhere
    \end{align*}
\end{proof}
 \paragraph{\textbf{The scalar equation}}
We conclude this section by recalling some known properties of the scalar equation \eqref{scalar eq}. Let $\Psi_\lambda:H^1(\Omega)\to\R$ be the associated energy functional:
\[\Psi_\lambda(u)=\frac12\int_\Omega|\nabla u|^2+\frac{\lambda}{2}\int_\Omega u^2-\frac14\int_\Omega u^4.\]
We denote by $L_\lambda$ the least energy level:
\[L_\lambda:=\inf\left\{ \Psi_\lambda(u):\ u\in H^1(\Omega)\setminus\{0\},\ \Psi_\lambda'(u)=0 \right\}.\]

Reasoning as in \cite[Proposition 1.2]{saldana2022least}, one can also write the least energy as follows
\begin{equation}\label{caratterizzazione les quoziente}
2\sqrt{L_\lambda}=\inf_{\substack{u\in H^1(\Omega)\\|u|_4=1}}\left(\int_\Omega|\nabla u|^2+\lambda\int_\Omega u^2\right),\qquad\text{if $\lambda>0$.}    
\end{equation}

As mentioned in the introduction, \eqref{scalar eq} admits the constant solutions $u\equiv \pm\sqrt{\lambda}$ for $\lambda>0$, and the problem \eqref{scalar eq} has a non-constant least energy  (ground state) solution $u\in C^2(\Omega)\cap C^1(\bar\Omega)$ for $\lambda>0$ large. On the other hand, we have existence of least energy solutions also when $\lambda\le0$, but in this case we do not have constant solutions.  

\begin{lemma}\label{ground state scalar}
    Let $N=4$. The level $L_\lambda$ is achieved and
    \begin{equation}\label{threshold scalar equation}
  L_{\lambda} < \frac{S^2}{8}.
\end{equation}
\end{lemma}
\begin{proof}
    The case $\lambda>0$ follows from \cite{adimurthimancini1}. When $\lambda=0$, we refer the reader to \cite{comte1991existence, saldana2022least}. If $\lambda<0$, one can adapt \cite{szulkin2010method, SzulkinWethWillem} (which deal with a Dirichlet problem, so that the upper bound therein is $S^2/4$) to the Neumann case.
\end{proof}

\begin{remark}
Let $N\leq 4$. As in \cite[Chapter 4]{szulkin2010method} and \cite{SzulkinWethWillem}, the least energy level is characterized as follows:
    \begin{equation}\label{caratterizzazione l.e.s. scalare}
    L_\lambda=\inf_{\substack{t>0\\ \tilde w\in\tilde{H}_\lambda}}\Psi_\lambda(tu+\tilde w),
\end{equation}
where $\tilde H_\lambda=\text{span}\{\varphi_1,\dots,\varphi_k\}$ for $\lambda\le-\mu_k(\Omega)$, where, for every $j\ge1$, $\mu_j(\Omega)=\mu_j^N(\Omega)$ is the j-th eigenvalue of the Laplacian with Neumann boundary conditions and $\varphi_j$ is the associated eigenfunction. 
\end{remark}

\noindent
Furthermore, we have 
\begin{equation}\label{stima ground state scalare}
 L_{\lambda} \leq M \lambda^{\frac{4-N}{2}}
\end{equation}
for all \(\lambda > 0\) and $N\le4$, where $M$ is as in \eqref{const thm 1.1}. Indeed, if \(u\) is a solution of \(-d \Delta u + u = u^3\), then the rescaled function \(v = d^{-1/2} u\) solves \(-\Delta v + d^{-1} v = v^3\). Taking \(d = \lambda^{-1}\), the estimate follows from \cite[Proposition 1.2]{lin1988large} if $N \le 3$. For  $N=4$, the bound from above is given by \eqref{threshold scalar equation}.

\section{Cooperative case}\label{cooperative}
In this section we deal with the case $\beta>0$, proving Theorems \ref{thm ground state}, \ref{Theorem 1.1}, and Proposition \ref{sharp beta*}.

\subsection{Existence of ground states}
Here, we show the  Theorem \ref{thm ground state}, namely that the ground state level is always achieved in the cooperative case $\beta>0$, for any $\lambda_1,\lambda_2\in\R$.
Following the notation of \cite{clapp2022solutions}, we define, for $\boldsymbol u=(u_1,u_2)$ and ${\boldsymbol v}=({v}_1,{v}_2)\in H$,
\begin{align*}
   B(\boldsymbol u,\boldsymbol v)=B((u_1,u_2),(v_1,v_2)):=B_1(u_1,v_1)+B_2(u_2,v_2),\\
    B_i(u_i,v_i):=\int_\Omega\nabla u_i\cdot\nabla v_i+\lambda_i\int_\Omega u_iv_i,\ \ i=1,2.
\end{align*}
Given $i=1,2$, we split the space $H^1(\Omega)$ as
\[
H^1(\Omega)=\tilde H_i \oplus H^+_i,\quad \text{ with } \quad  \tilde H_i=H_i^-\oplus H_i^0,
\]
where $H^-_i,H^0_i,H^+_i$ are, respectively, the subspaces of $H^1(\Omega)$ where the bilinear form $B_i$ is negative, null and positive (the spaces depend on $\lambda_i$, but we omit this dependence for simplicity of notation). For
\[
H^-:=H_1^-\times H_2^-,\quad H^0:=H_1^0\times H_2^0, \quad H^+:=H_1^+\times H_2^+ \quad \text{ and } \tilde H:=H^-\oplus H^0,
\]
we obtain the decomposition
\[
H=\tilde H\oplus H^+=H^-\oplus H^0\oplus H^-.
\]

\begin{remark}
Denoting by \(\mu_j\) the \(j\)-th eigenvalue of the Laplacian with Neumann boundary conditions (counting multiplicities), and by \(\varphi_j\) an associated eigenfunction, for $j\ge1$, we have 
   \begin{itemize}
       \item[$(i)$] $\lambda_i>0\implies$ $H_i^+=H^1(\Omega)$, $\tilde H_i=\{0\}$.
       \item[$(ii)$]  $\lambda_i\le\mu_j\implies \tilde H_i=H_i^0\oplus H_i^-=span\{\varphi_1,\dots,\varphi_j\}$ and $H_i^+$ is its orthogonal complement.
   \end{itemize}
\end{remark}
Recall that the solutions of \eqref{Pb} are critical points of the $C^2-$energy functional:
\[\mathcal J_\beta(u,v)=\frac12B_1(u,u)+\frac12B_2(v,v)-\frac14\int_{\Omega}(u^4+\beta u^2v^2+v^4).\]

We use the following linking theorem:

\begin{theorem}[{{\cite{rabinowitz78linking}}}]\label{linking theorem}
    Let $H$ be a Hilbert space such that $H=V\oplus W$ with $dim(V)<+\infty$ and $V,W$ are closed. Assume that $J\in C^1(H;\R)$ and:
    \begin{itemize}
        \item there exist $\alpha,\rho>0$ such that $J\ge\alpha>0$ on $ W\cap \partial B_\rho$,
        \item there exist $R>\rho$ and $e\in W$ with $\|e\|=1$ such that $J\le0$ on $\partial Q$, where 
\[Q:=\left\{v+tRe\ :\  v\in V\cap\overline B_R,\ t\in[0,1]\right\}\]
\item $J$ satisfies the $(PS)$ condition at level $c$, where
\[c:=\inf_{h\in \Gamma}\sup_{w\in Q}J(h(w)),\ \ \Gamma:=\left\{h\in C(Q;H)\ :\ h\equiv id\ on\ \partial Q\right\}.\]
    \end{itemize}
    Then there exists a critical point for $J$ at level $c\ge\alpha>0$.
\end{theorem}

We apply Theorem \ref{linking theorem} with $J=\mathcal J_\beta$, $V=\tilde H$ (which is finite dimensional), $W=H^+$ and 
\[
Q=Q_{\boldsymbol e,R}:=\{ \tilde{\boldsymbol{w}}+tR\boldsymbol{e}\ :\ \tilde{\boldsymbol{w}}\in \tilde H\cap\overline B_R,\ t\in[0,1] \},\ \ \text{ for some $\mathbf{e}\in H^+\cap \partial B_1$ and $R>0$ large.}
\]

We denote by $c_\beta=c_\beta(\boldsymbol e)$ and $\Gamma_\beta$ the linking level and the minimax class, respectively.

 \begin{lemma}\label{linking geom}
     There exist $\alpha_\beta,\rho>0$, 
     \begin{itemize}
         \item[$(i)$] $\mathcal J_\beta\ge\alpha_\beta>0$ on $ H^+\cap\partial B_\rho$.
         \item[$(ii)$] For every  $\textbf{e}\in H^+\cap\partial B_1$, there exists $R_{\boldsymbol e}>\rho$ such that $\mathcal J_\beta\le0$ on $\partial Q_{\boldsymbol e,R}$ for every $R\ge R_{\boldsymbol e}$ and $\beta>0$. 
     \end{itemize}
 \end{lemma}
 \begin{proof}
 $(i)$ For any $\boldsymbol u=(u_1,u_2)\in H^+$ we have that  
 \begin{align*}
     \mathcal J_\beta(u_1,u_2)&=\frac12B_1(u_1,u_1)+\frac{1}{2}B_2(u_2,u_2)-\frac14\int_{\Omega}(u_1^4+2\beta u_1^2u_2^2+u_2^4)\\
     &\ge\frac12B_1(u_1,u_1)+\frac{1}{2}B_2(u_2,u_2)-\frac{1+\beta}{4}(|u_1|_4^4+|u_2|_4^4)\\
     &\ge\frac12B_1(u_1,u_1)+\frac{1}{2}B_2(u_2,u_2)-C_\beta(\|u_1\|^4+\|u_2\|^4).
 \end{align*}
 If $\boldsymbol u\in H^+\cap\partial B_\rho$, we get that $\mathcal J_\beta(u_1,u_2)\ge\alpha_\beta>0$ for some $\alpha_\beta,\rho>0$.\\
$(ii)$
Let $\boldsymbol e\in H^+$ with $\|\boldsymbol e\|=1$ and 
\[
\boldsymbol z\in\{\tilde{\boldsymbol w}+t\boldsymbol e\ :\  \tilde{\boldsymbol w}\in\tilde H\cap \partial B_1,\ t\in[0,1]\}\cup\{\tilde{\boldsymbol w}+\boldsymbol e\ :\  \tilde{\boldsymbol w}\in\tilde H\cap B_1\}=Z_1\cup Z_2.
\]
In particular, $\|z\|\le2$ and $\dim (Z_i)<+\infty$ for $i=1,2$. Thus,
 \begin{equation*}\label{eq geometria1}
     \mathcal J_\beta(R\boldsymbol z)\le\frac {R^2}{2}B(\boldsymbol z,\boldsymbol z)-\frac{R^4}{4}\int_{\Omega}(z_1^4+z_2^4)\le0
 \end{equation*}
 for $R\ge R_{\boldsymbol e}$ and $R_{\boldsymbol e}>\rho$ (independent of $\beta$) large enough. Moreover, for any $\tilde{\boldsymbol{w}}=(\tilde w_1,\tilde w_2)\in \tilde H$ we have:
 \begin{equation*}\label{eq geometria}
 \mathcal J_\beta(\tilde{\boldsymbol{w}})\le \frac{1}{2}B(\tilde{\boldsymbol w},\tilde{\boldsymbol w})-\frac{1}{4}\int_{\Omega}(\tilde{w}_1^4+\tilde w_{2}^4)\le0,
 \end{equation*}
 since $B(\cdot,\cdot)$ is negative (semi-)definite in $\tilde H$.  Therefore, $\mathcal J_\beta\le0$ on $\partial Q_{\boldsymbol e,R}$ for every $R\ge R_{\boldsymbol e}$.
 \end{proof}

Next, we prove that $\mathcal J_\beta$ satisfies the $(PS)$-condition. 
In the critical case, when the dimension is \( N = 4 \), Cherrier's inequality (Lemma~\ref{Cherrier}) plays the same role as the Sobolev inequality in the Dirichlet case and is crucial in the following compactness result. In the subcritical regime, on the other hand, the Neumann and Dirichlet cases have essentially the same proof.

\begin{theorem}\label{PS condition beta>0}
Let $N\leq 4$. Then every Palais-Smale sequence is bounded. Moreover:
\begin{itemize}
\item if $N\leq 3$, the functional $\mathcal J_\beta$ satisfies the $(PS)_c$-condition for every $c\in \R$;
\item if $N=4$, the functional $\mathcal J_\beta$ satisfies the $(PS)_c$-condition whenever
 $c<\min\left\{\frac{S^2}{4(1+\beta)},\frac{S^2}{8}\right\}$. 
\end{itemize}
\end{theorem}
\begin{proof}
 The proof that every Palais-Smale sequence $\{(u_n,v_n)\}$ is bounded follows as in \cite[Lemma 3.2]{clapp2022solutions}. Then we can assume that $u_n\wto u$ and $v_n\wto v$ in $H^1(\Omega)$ and, in particular,
 for any $\varphi\in H^1(\Omega)$ we have
\begin{equation}\label{Jb'=0 eq1}
    \begin{aligned}
      \int_{\Omega}\nabla u_n\cdot\nabla \varphi\to \int_{\Omega}\nabla u\cdot\nabla\varphi,\qquad 
      \int_{\Omega}\nabla v_n\cdot\nabla \varphi\to \int_{\Omega}\nabla v\cdot\nabla\varphi
    \end{aligned}
\end{equation}
and
\begin{align*}
\langle\mathcal J'_\beta(u_n,v_n),(\phi,\psi)\rangle=\langle\mathcal J'_\beta(u,v),(\phi,\psi)\rangle+o(1),\quad \forall\ (\phi,\psi)\in H.
\end{align*}
Recalling that $\mathcal J_\beta'(u_n,v_n)=o(1)$ and passing to the limit as $n\to+\infty$, we get $\mathcal J_\beta'(u,v)=0$.
\smallbreak

When $N\le3$, the converge is strong in $L^p(\Omega)$ for $p\in[1,4]$ and
\begin{equation}\label{compactness subcritical}
    o(1)=\langle\mathcal J_\beta'(u_n,v_n)-\mathcal J_\beta'(u,v),(u_n-u,v_n-v)\rangle=\|(u_n-u,v_n-v)\|^2+o(1).
\end{equation}
This implies the strong convergence in $H^1(\Omega)$ and the $(PS)$-condition.

\smallbreak Now, we focus on the critical case $N=4$. We have $u_n\to u, v_n\to v$ strongly  in $L^q(\Omega)$ for $q\in[1,4)$ and weakly for $q=4$. We define $\theta_n:=u_n-u, \sigma_n:=v_n-v$. By the Brezis-Lieb Lemma \cite{brezis1983relation}:
\begin{equation}\label{brezis lieb eq1}
\begin{aligned}
    |u_{n}|_4^4=|u|_4^4+|\theta_{n}|_4^4+o(1),\qquad |v_n|_4^4=|v|_4^4+|\sigma_{n}|_4^4+o(1);
    \end{aligned}
\end{equation}
exploiting the weak convergence $\sigma_n,\theta_n\wto0$ in $H^1(\Omega)$,
\begin{equation}\label{Jb'=0 eq2}
\begin{aligned}
\int_\Omega u_n^2v_n^2&=\int_\Omega (\theta_n+u)^2(\sigma_n+v)^2=\int_\Omega \theta_n^2 \sigma_n^2 +\int_\Omega 2 \theta_n^2 \sigma_n v +\int_\Omega \theta_n^2 v^2\\
&+\int_\Omega 2 \theta_n u \sigma_n^2 +\int_\Omega 4 \theta_n u \sigma_n v +\int_\Omega 2 \theta_n u v^2 +\int_\Omega u^2 \sigma_n^2 +\int_\Omega 2 u^2 \sigma_n v +\int_\Omega u^2 v^2\\
&=\int_\Omega\theta_n^2\sigma_n^2+\int_\Omega u^2v^2+o(1).
\end{aligned}
\end{equation}
Using \eqref{Jb'=0 eq1}, \eqref{brezis lieb eq1} and \eqref{Jb'=0 eq2}, we have
\begin{equation}\label{nehari equation+o(1)}
\begin{split}
|\nabla\theta_n|_{2}^2&-\int_{\Omega}(\theta_n^4+\beta\theta_n^2\sigma_n^2)=|\nabla u_n|_{2}^2-|\nabla u|_{2}^2-\int_\Omega u_n^4+\int_\Omega u^4 \\
&-\beta\int_\Omega u_n^2v_n^2+\beta\int_\Omega u^2v^2+o(1)=o(1),\\
|\nabla\sigma_n|_{2}^2&-\int_\Omega(\sigma_n^4+\beta\theta_n^2\sigma_n^2)=|\nabla v_n|_{2}^2-|\nabla v|_{2}^2-\int_\Omega v_n^4+\int_\Omega v^4\\&-\beta\int_\Omega u_n^2v_n^2+\beta\int_\Omega u^2v^2+o(1)=o(1)
\end{split}
\end{equation}
where the last identity comes from $\mathcal J_\beta'(u,v)=0$ and $\mathcal J_\beta'(u_n,v_n)=o(1)$.
We can derive also that 
\begin{equation}\label{eq. energia}
\begin{split}
c+o(1)=\mathcal J_\beta(u_n,v_n)&=\frac14\int_\Omega(|\nabla u_n|^2+|\nabla v_n|^2)+o(1)\\
&=\frac14B((u,v),(u,v))+\frac14\int_{\Omega}(|\nabla\theta_n|^2+|\nabla\sigma_n|^2)+o(1)\\
&\ge\frac14\int_{\Omega}(|\nabla\theta_n|^2+|\nabla\sigma_n|^2)+o(1).
\end{split}
\end{equation}
Passing to a subsequence, we may assume that the following limits are finite:
\[\xi_1:=\lim_{n\to+\infty}\int_{\Omega}|\nabla\theta_n|^2,\ \ \xi_2:=\lim_{n\to+\infty}\int_{\Omega}|\nabla\sigma_n|^2.\]
If $\xi_1=\xi_2=0$, we obtain the strong convergence of $u_n,v_n$ to $u,v$ in $H^1(\Omega)$.
Suppose now that $\xi_1+\xi_2>0$.

 Cherrier's inequality (Lemma \ref{Cherrier}) implies that, for each $\varepsilon>0$,
\begin{align*}
    \int_\Omega\theta_n^4\le\left(\frac{\sqrt2}{S}+\varepsilon\right)^2|\nabla\theta_n|_2^4+o(1),\quad  \int_\Omega\sigma_n^4\le\left(\frac{\sqrt2}{S}+\varepsilon\right)^2|\nabla\sigma_n|_2^4+o(1).
\end{align*}
Thus, by \eqref{nehari equation+o(1)} and H\"older's inequality,

\begin{align*}
  |\nabla\theta_n|_2^2&=\int_\Omega(\theta_n^4+\beta\theta_n^2\sigma_n^2)+o(1)\le |\theta_n|_4^4+\beta|\theta_n|_4^2\cdot|\sigma_n|_4^2+o(1)\\
  &\le\left(\frac{\sqrt2}{S}+\varepsilon\right)^2\cdot(|\nabla\theta_n|_2^4+\beta|\nabla\theta_n|_2^2\cdot|\nabla\sigma_n|_2^2)+o(1),\\
|\nabla\sigma_n|_2^2&=\int_\Omega(\sigma_n^4+\beta\theta_n^2\sigma_n^2)+o(1)\le(|\sigma_n|_4^4+\beta|\theta_n|_4^2\cdot|\sigma_n|_4^2)+o(1)\\
    &\le\left(\frac{\sqrt2}{S}+\varepsilon\right)^2\cdot(|\nabla\sigma_n|_2^4+\beta|\nabla\theta_n|_2^2\cdot|\nabla\sigma_n|_2^2)+o(1).
\end{align*}
As $n\to+\infty$,
\[
\xi_1\le\left(\frac{\sqrt2}{S}+\varepsilon\right)^2(\xi_1^2+\beta\xi_1\xi_2)\qquad \xi_2\le\left(\frac{\sqrt2}{S}+\varepsilon\right)^2(\xi_2^2+\beta\xi_1\xi_2).
\]
If $\xi_1,\xi_2>0$, we have that $\xi_1+\xi_2\ge\frac{S^2}{1+\beta}$ letting $\varepsilon\to0^+$. This contradicts $c<\frac{S^2}{4(1+\beta)}$. Indeed, dividing by $\xi_1$ the first equation and by $\xi_2$ the second one, we get:
\[
\begin{cases}
1\le\left(\frac{\sqrt2}{S}+\varepsilon\right)^2(\xi_1+\beta \xi_2)\\
1\le\left(\frac{\sqrt2}{S}+\varepsilon\right)^2(\xi_2+\beta \xi_1)
\end{cases}
\implies \xi_1+\xi_2\ge\frac{2}{1+\beta}\left(\frac{\sqrt2}{S}+\varepsilon\right)^{-2}
\]
If $\xi_1=0$, then $\xi_2\ge\frac{S^2}{2}$ and
    \[c=\lim_{n\to+\infty}\mathcal J_\beta(u_n,v_n)\ge\frac{\xi_2}{4}\ge\frac{S^2}{8},\]
    from \eqref{eq. energia}. The hypothesis $c<\frac{S^2}{8}$ implies the contradiction $\frac{S^2}{8}>c\ge\frac{S^2}{8}$.
    
 The case $\xi_2=0$ is similar.
\end{proof}

In the critical case, in order to show the compactness condition of Theorem \ref{PS condition beta>0}, we follow \cite[proof of Theorem 1.2]{clapp2022solutions}, which deals with the Dirichlet case. In the Neumann case,  the estimate in Lemma \ref{Estimates for C} allows us to obtain compactness for  \emph{all} values of $\lambda_1,\lambda_2$, since we concentrate the bubbles at a point on the boundary with positive curvature. In the Dirichlet case, one concentrates the bubbles inside the domain, and the condition for $N=4$ is satisfied only when $-\lambda_1,-\lambda_2$ are not  eigenvalues. Recalling Lemma \ref{linking geom}, we have the following.
\begin{proposition}\label{soglia}
Let $N=4$. There exists $\boldsymbol{e}\in H^+\cap\partial B_1$  such that, by choosing $R=R_{\boldsymbol{e}}$, we have $c_\beta=c_\beta(\boldsymbol e)<\frac{S^2}{4(1+\beta)}$ for every $\beta>0$.
\end{proposition}
\begin{proof}
We consider $w_\varepsilon=U_{\varepsilon,0}(|x|)\eta(x)$, as in Lemma \ref{Estimates for C} and $\boldsymbol z=\tilde {\boldsymbol w}+t{\boldsymbol w}_\varepsilon$, with $\tilde {\boldsymbol w}=(\tilde w_1,\tilde w_2)\in\tilde H$ and  ${\boldsymbol w}_\varepsilon=(w_\varepsilon,w_\varepsilon)$. We estimate $\max_{\substack{\tilde {\boldsymbol w}\in\tilde H, t>0}}\mathcal J_\beta(\boldsymbol z).$ 
 Since $\beta>0$, using Lemma \ref{Estimates for C} and equations $(4.8)-(4.9)$ in \cite[p. 14]{clapp2022solutions}, it follows that
 \begin{align*}
   \mathcal J_\beta(\tilde w_1+tw_\varepsilon,\tilde w_2+tw_\varepsilon)\le c_1\left[t^2+t^3(\|\tilde w_1\|+\|\tilde w_2\|)\sqrt\varepsilon+t(\|\tilde w_1\|+\|\tilde w_2\|)\right]-c_2(t^4+\|\tilde w_1\|^4+\|\tilde w_2\|^4),
 \end{align*}
 for some $c_1,c_2>0$.
 Hence, there exists $R>0$ large such that $\mathcal J_\beta(\tilde w_1+tw_\varepsilon,\tilde w_2+tw_\varepsilon)\le0$ if either $t\ge R$ or $\|\tilde{\boldsymbol w}\|\ge R$. Thus, we can assume that $\|\tilde{\boldsymbol w}\|\le R$ and $t\le R$.

Notice that, since the equations $(4.8)-(4.10)$ in \cite[p. 14]{clapp2022solutions} hold also in our case, we can deduce the existence of $k_1,k_2>0$ such that
\begin{align*}
      \mathcal J_\beta(\tilde w_1+tw_\varepsilon,\tilde w_2+tw_\varepsilon)&\le\frac{t^2}{2}B_1(w_\varepsilon,w_\varepsilon)+\frac{t^2}{2}B_2(w_\varepsilon,w_\varepsilon)-\frac{1+\beta}{2}t^4\int_\Omega w_\varepsilon^4\\&-t^3\int_{\Omega} (\tilde w_1+\tilde w_2)w_\varepsilon^3 -\beta t^3\int_{\Omega} (\tilde w_1+\tilde w_2)w_\varepsilon^3 -2\beta t^2\int_{\Omega} \tilde w_1\tilde w_2w_\varepsilon^2\\
       &\quad -k_2(\|\tilde w_1\|^4+\|\tilde w_2\|^4) -\frac{\beta k_1}{2} \|\tilde w_1\|^2 \|\tilde w_2\|^2+tB_1(\tilde w_1,w_\varepsilon)+tB_2(\tilde w_2,w_\varepsilon).
      \end{align*}
   
     By Lemma \ref{Estimates for C}, together with the fact that \( \tilde H \) is finite-dimensional and $\|\tilde{\boldsymbol w}\|,t\le C$, we obtain

      \begin{align*}
      \mathcal J_\beta(\tilde w_1+tw_\varepsilon,\tilde w_2+tw_\varepsilon)&\le t^2 \left(\frac{S^2}{2}-C_1\sqrt\varepsilon+o(\sqrt\varepsilon)\right) -\frac{(1+\beta)t^4}{2} \left(\frac{S^2}{2}-C_2\sqrt\varepsilon+o(\sqrt\varepsilon)\right)\\
      &\quad +O(\sqrt\varepsilon)(\|\tilde w_1\|+\|\tilde w_2\|) + O(\varepsilon|\log\varepsilon|)\|\tilde w_1\|\|\tilde w_2\|\\
      &\quad -k_2(\|\tilde w_1\|^4+\|\tilde w_2\|^4) -\frac{\beta k_1}{2}\|\tilde w_1\|^2 \|\tilde w_2\|^2+o(\sqrt\varepsilon). 
\end{align*}

   From $(4.12)-(4.13)$ in \cite[p. 15]{clapp2022solutions}, we infer that

   \[O(\sqrt\varepsilon)(\|\tilde w_1\|+\|\tilde w_2\|)-k_2(\|\tilde w_1\|^4+\|\tilde w_2\|^4)\le C\varepsilon^{2/3}=o(\sqrt\varepsilon),\]
   \[O(\varepsilon|\log\varepsilon|)\|\tilde w_1\|\cdot\|\tilde w_2\|-\frac{\beta k_1}{2}\|\tilde w_1\|^2\cdot\|\tilde w_2\|^2\le C\varepsilon^2\log^2\varepsilon=o(\sqrt\varepsilon).\]
   Thus,
   \begin{align*}
       \mathcal J_\beta(\tilde w_1+tw_\varepsilon,\tilde w_2+tw_\varepsilon)\le t^2 \left(\frac{S^2}{2}-C_1\sqrt\varepsilon+o(\sqrt\varepsilon)\right)-\frac{(1+\beta)t^4}{2}\left(\frac{S^2}{2}-C_2\sqrt\varepsilon+o(\sqrt\varepsilon)\right)+o(\sqrt\varepsilon).
   \end{align*}
   
Now, taking the maximum over all $t>0$ and $\tilde{\boldsymbol w}\in\tilde H$, we obtain:
\[
     \max_{\substack{\tilde{\boldsymbol w}\in\tilde H\\t>0}}\mathcal J_\beta(\tilde w_1+tw_\varepsilon,\tilde w_2+tw_\varepsilon)\le\frac{\left(\frac{S^2}{2}-C_1\sqrt\varepsilon+o(\sqrt\varepsilon)\right)^2}{2(1+\beta)\left(\frac{S^2}{2}-C_2\sqrt\varepsilon+o(\sqrt\varepsilon)\right)}+o(\sqrt\varepsilon)<\frac{S^2}{4(1+\beta)},
     \]
     for $\varepsilon>0$ small enough (recall Lemma \ref{Estimates for C}, in particular the fact that $C_1>C_2$).
  Taking    $\boldsymbol{e}_\varepsilon:=\frac{\boldsymbol{w}_\varepsilon}{\|\boldsymbol{w}_\varepsilon\|}$ 
, we obtain $c_\beta(\boldsymbol{e}_\varepsilon)<\frac
     {S^2}{4(1+\beta)}$. 
\end{proof}
\begin{remark}
     We can also write the threshold value of Proposition \ref{soglia} as half of the compactness threshold value in the Dirichlet case \cite[eq. $(4.6)$]{clapp2022solutions}: \[\frac{S^2}{4(1+\beta)}=\frac{1}{2N}S_{\infty,\beta}^{N/2},\]
     where $N=4$ and $S_{\infty,\beta}$ is defined in  \eqref{Sinfbeta}.
\end{remark}

\begin{lemma}\label{cb critico}
One has that $c_\beta$, introduced in Proposition \ref{soglia}, is a critical level if $N \le 3$, or if $N=4$ and $\beta >1$. 
\end{lemma}
\begin{proof}
According to Lemma~\ref{linking geom}, Theorem~\ref{PS condition beta>0}, and Proposition~\ref{soglia}, we can apply Theorem~\ref{linking theorem}. Therefore, there exists a nontrivial pair \( (u,v) \in H \setminus \{(0,0)\} \) such that \( \mathcal{J}_\beta(u,v) = c_\beta \) and \( \mathcal{J}_\beta'(u,v) = 0 \), i.e., \( c_\beta \) is a critical level for \( \mathcal{J}_\beta \).
\end{proof}

 \begin{remark}\label{Remark Mountain Pass}

    In the particular case $\lambda_1,\lambda_2>0$, one can prove by similar arguments that the Mountain Pass level:   
\[
\widetilde c_\beta:=\inf_{\boldsymbol u\in H\setminus\{\boldsymbol 0\}}\max_{t>0}\mathcal J_\beta(t\boldsymbol u)
\]
is attained for $\beta>1$. Moreover, adapting \cite[Theorem 4.2]{willem2012minimax} to the Neumann case, it follows that $\widetilde c_\beta=g_\beta$.
\end{remark}

\begin{proof}[Proof of Theorem \ref{thm ground state}]
Let $u_i\in H^1(\Omega)$ be a solution of \eqref{scalar eq} for $i=1,2$. Then
    \[(u_1,0),(0,u_2)\in\{(u,v)\in H\setminus\{(0,0)\}\ :\ \mathcal J_\beta'(u,v)=0\}\ne\emptyset,\]
   and we can take a minimizing sequence $\{(u_n,v_n)\}$  for $g_\beta$. In particular, $\{(u_n,v_n)\}$ is a Palais-Smale sequence at level $g_\beta$. If $\beta>1$, $g_\beta<\frac{S^2}{4(1+\beta)}$ according to Lemma \ref{cb critico} and Proposition \ref{soglia}, while if $\beta\in(0,1)$, $g_\beta\le\min\{L_{\lambda_1},L_{\lambda_2}\}<\frac{S^2}{8}$ by \eqref{threshold scalar equation}.
    By  Theorem~\ref{PS condition beta>0},  we can deduce that there exists $(u,v)\in H$ such that $u_n\to u,v_n\to v$ strongly in $H^1(\Omega)$, up to a subsequence. Since $\beta>0$, we can follow \cite[Lemma 3.3]{clapp2022solutions} to get that $(u,v)\not\equiv(0,0)$ and the proof of Theorem \ref{thm ground state} is complete.   
\end{proof}

We conclude the section proving Proposition \ref{sharp beta*}:
\begin{proof}[Proof of Proposition \ref{sharp beta*}]
We follow \cite{mandel2015minimal}, adapting it to the bounded domain case. Let
\[\Phi(\boldsymbol u):=\frac{(4L)^{-1}B(\boldsymbol u,\boldsymbol u)^2-|u_1|_4^4-|u_2|_4^4}{2|u_1u_2|_2^2},\quad \boldsymbol u=(u_1,u_2)\in H,\ u_1\cdot u_2\not\equiv 0,
 \]
 with $L=\min\{L_{\lambda_1},L_{\lambda_2}\}$. 

 By \eqref{caratterizzazione les quoziente}, we have that
 \[\int_\Omega|\nabla u|^2+\lambda_i\int_\Omega u^2\ge 2\sqrt{L_{\lambda_i}}\left(\int_\Omega u^4\right)^{\frac12}, \qquad \text{for every $u\in H^1(\Omega)$, $i=1,2$.}\]
Now, by Hold\"er inequality, for $r=\frac{|u_1|_4}{|u_2|_4}$,  we have 
\begin{align*}
    \Phi(\boldsymbol u)&\ge\frac{(4L)^{-1}B(\boldsymbol u,\boldsymbol u)^2-|u_1|_4^4-|u_2|_4^4}{2|u_1|_4^2|u_2|_4^2}=\frac{\frac{(4L)^{-1}}{|u_2|_4^4}B(\boldsymbol u,\boldsymbol u)^2-r^4-1}{2r^2}\\
    &\ge\frac{\frac{(4L)^{-1}}{|u_2|_4^4}(2\sqrt{L_{\lambda_1}}|u_1|_4^2+2\sqrt{L_{\lambda_2}}|u_2|_4^2)^2-r^4-1}{2r^2}\\
    &=\frac{L^{-1}L_{\lambda_1}r^4+L^{-1}L_{\lambda_2}+2L^{-1}\sqrt{L_{\lambda_1} L_{\lambda_2}}r^2-r^4-1}{2r^2}\\
    &\ge\frac{r^4+1+2L^{-\frac12}L^{-\frac12}\sqrt{L_{\lambda_1}L_{\lambda_2}}r^2-r^4-1}{2r^2}\ge1.
\end{align*}
 Hence, $\beta_*\ge1$.\\ 
 $(i)$ Let us assume $\beta<\beta_*$ and by contradiction let $\boldsymbol u^*$ be a fully non-trivial ground state solution. Since $\beta_*\le \Phi(\boldsymbol u^*)$, $\beta<\Phi(\boldsymbol{u}^*)$ and we have:
\[\mathcal J_\beta(\boldsymbol u^*)=\frac14\cdot\frac{B(\boldsymbol u^*,\boldsymbol u^*)^2}{|u_1^*|_4^4+2\beta|u_1^*u_2^*|_2^2+|u_2^*|_4^4}>\frac14\cdot\frac{B(\boldsymbol u^*,\boldsymbol u^*)^2}{|u_1^*|_4^4+2\Phi(\boldsymbol u^*)|u_1^*u_2^*|_2^2+|u_2^*|_4^4}=\min\{L_{\lambda_1},L_{\lambda_2}\}.\]
This is a contradiction.\\
$(ii)$ On the other hand, if $\beta>\beta_* \ge 1$, there exists a sequence \[\{\boldsymbol u_n\}\subset\{\boldsymbol u\in H\ :\ u_1u_2\not\equiv0\}\] such that $\Phi(\boldsymbol u_n)\le\beta-\frac1n<\beta$, for every $n\ge1$. Let $t>0$ be the maximizer of the function $s\mapsto\mathcal J_\beta(s\boldsymbol u_n)$. Then, $t\boldsymbol u_n$ satisfies:
\[B(t\boldsymbol u_n,t\boldsymbol u_n)=\int_\Omega|tu_{n,1}|^4+2\beta\int_\Omega (tu_{n,1})^2(tu_{n,2})^2+\int_\Omega(tu_{n,2})^4.\]
Thus, $\Phi(t\boldsymbol u_n)=\Phi(\boldsymbol u_n)<\beta$ and 
\[4\mathcal J_\beta(t\boldsymbol u_n)=\frac{B(\boldsymbol u_n,\boldsymbol u_n)^2}{|u_{n,1}|_4^4+2\beta|u_{n,1}u_{n,2}|_2^2+|u_{n,2}|_4^4}<4\min\{L_{\lambda_1},L_{\lambda_2}\},\]
by the same computations of the previous step. 
Hence, the Mountain Pass level $\widetilde c_\beta$ defined in Remark \ref{Remark Mountain Pass} satisfies $\widetilde c_\beta<\min\{L_{\lambda_1},L_{\lambda_2}\}$ and any ground-state solution is a least energy solution.
\end{proof}

\subsection{Strongly cooperative case}
 We prove Theorem \ref{Theorem 1.1}-$(ii)$.

\begin{proof}[Proof of Theorem \ref{Theorem 1.1}-$(ii)$]

Let $(u,v)$ be a ground state solution of \eqref{Pb}, which exists by Theorem \ref{thm ground state}.

\textbf{Step 1:} A ground state solution is a least energy solution for $\beta>1$ large enough.

If $N=4$, $g_\beta\le c_\beta<\frac{S^2}{4(1+\beta)}$ by Proposition \ref{soglia}, so $g_\beta\to0$ as $\beta\to+\infty$. Hence, $g_\beta<\min\{L_{\lambda_1},L_{\lambda_2}\}$ for $\beta>1$ large enough (where $L_{\lambda_i}$ is the least energy level of the scalar equation $-\Delta u+\lambda_iu=u^3$ with $u\in H^1(\Omega)$ and $i=1,2$), therefore $u,v\not\equiv0$. 

Now, we focus on the subcritical case $N\le3$. Let $j_1,j_2\ge1$ be such that $\tilde H_i=\text{span}\{\varphi_1,\dots,\varphi_{j_i}\}$, and let $\boldsymbol e=(e_1,e_2)\in H^+$ with $e_i=s_i\varphi_{j_i+1}$, $s_i\in\R$ such that $\|\boldsymbol e\|=1$ and $R=R_e$  be as in Lemma \ref{linking geom}.
We know that \[c_\beta(\boldsymbol e)\le\sup_{\overline{Q}}\mathcal J_\beta=\sup_{\substack{\tilde{\boldsymbol w}\in\tilde H\cap\overline B_R\\ t\in[0,1]}}\mathcal J_\beta(\tilde{\boldsymbol w}+tR\boldsymbol e),\]
with maximum achieved at some $\tilde{\boldsymbol w}^\beta+t_\beta R \boldsymbol e\in Q$. Since $\{t_\beta\}$ and $\{\tilde{\boldsymbol w}^\beta\}$ are uniformly bounded and $\tilde H$ is finite dimensional, we can assume that $\tilde{\boldsymbol w}^\beta\to\tilde{\boldsymbol w}^\infty,t_\beta\to t_\infty$. If $t_\infty=0$, we have
\[
\mathcal J_\beta(\tilde{\boldsymbol w}^\beta+t_\beta R \boldsymbol e)\le\frac{t_\beta^2R^2}{2}B(\boldsymbol e,\boldsymbol e)\to0.
\]
Thus, $c_\beta\to0$ and we conclude as in the critical case. Let us now consider by contradiction the case $t_\infty>0$, namely we assume that there exists $\delta>0$ such that $t_\beta\ge\delta$ for $\beta$ large enough. We write the maximizer as follows:
\[
\tilde{\boldsymbol w}^\beta+t_\beta R \boldsymbol e=t_\beta\left(\frac{\tilde{\boldsymbol w}^\beta}{t_\beta}+R\boldsymbol e\right)=t_\beta(\tilde{\boldsymbol z}^\beta+R\boldsymbol e)
.\]
Since $\langle\mathcal J_\beta'(t_\beta(\tilde{\boldsymbol z}^\beta+R\boldsymbol e)),t_\beta(\tilde{\boldsymbol z}^\beta+R\boldsymbol e)\rangle=0$,

 \begin{equation}\label{bound t beta}
 t_\beta^2=\frac{B_1(Re_1+\tilde{z}_1^\beta,Re_1+\tilde{z}_1^\beta)+B_2(Re_2+\tilde{z}_2^\beta,Re_2+\tilde{z}_2^\beta)}{\displaystyle{\int_\Omega\left[(Re_1+\tilde{z}_1^\beta)^4+2\beta (Re_1+\tilde{z}_1^\beta)^2(Re_2+\tilde{z}_2^\beta)^2+(Re_2+\tilde{z}_2^\beta)^4\right]}} \le\frac{B_1(Re_1,Re_1)+B_2(Re_2,Re_2)}{\displaystyle{2\beta\int_\Omega(Re_1+\tilde{z}_1^\beta)^2(Re_2+\tilde{z}_2^\beta)^2}}.
 \end{equation}
Assume that
\begin{equation}\label{limit int}
\lim_{\beta\to+\infty}\int_\Omega (Re_1+\tilde{z}_1^\beta)^2(Re_2+\tilde{z}_2^\beta)^2=0.\end{equation}  Since $\{\tilde{z}_1^\beta\},\{\tilde{z}_2^\beta\}$ are uniformly bounded in $\beta$ (as $t_\beta\geq \delta$), we can assume that there exist $\tilde{z}_1^\infty,\tilde{z}_2^\infty$ such that $\tilde{z}_1^\beta\to \tilde{z}_1^\infty$ and $\tilde{z}_2^\beta\to \tilde{z}_2^\infty$ strongly (recall again that $\tilde H_i$ is finite dimensional). Now,  $Re_i+\tilde z_i^\infty\equiv0$ on a open set $\tilde\Omega\subset\Omega$ is not possible, by \cite[Lemma 3.3]{SzulkinWethWillem}. Then  $Re_i+\tilde z_i^\infty\equiv0$ in $\Omega$, but this implies that $Re_i\equiv\tilde z_i^\infty\equiv0$  since $\tilde H_i\cap H^+_i=\{0\}$ for $i=1,2$, a contradiction.  Thus \eqref{limit int} is not true and, by \eqref{bound t beta}, $t_\beta \to 0$ as $\beta \to +\infty$.

This step proves Theorem \ref{Theorem 1.1}-$(ii)$ when $(\lambda_1,\lambda_2)\notin(0,+\infty)^2$.\\
\textbf{Step 2:}
Every ground state solution  is a non-constant least energy solution if $\lambda_1,\lambda_2>0$, $\beta>\max\left\{\frac{\lambda_1}{\lambda_2},\frac{\lambda_2}{\lambda_1}\right\}$ and \eqref{eq least energy non constant} holds. \\ 
Let $(c_1, c_2)$ be a constant ground state solution. 
By \eqref{costanti ammissibili},
\[\mathcal J_\beta(c_1,c_2)=\frac{\lambda_1^2-2\beta\lambda_1\lambda_2+\lambda_2^2}{4(1-\beta^2)}|\Omega|.\]
By Sobolev's embedding and for every $\lambda>0$, and $u\in H^1(\Omega)$ a ground state solution for $-\Delta u+\lambda u=u^3$ in $H^1(\Omega)$, we obtain 
\begin{align*}
    \int_\Omega u^4&\le C_S^{-2}\left(\int_\Omega |\nabla u|^2+\int_\Omega u^2\right)^2\\
    &\le C_S^{-2}\left(\max\left\{1,\frac{1}{\lambda}\right\}\right)^2\left(\int_\Omega |\nabla u|^2+\lambda\int_\Omega u^2\right)^2\\
    &=C_S^{-2}\left(\max\left\{1,\frac{1}{\lambda}\right\}\right)^2\left(\int_\Omega |\nabla u|^2+\lambda\int_\Omega u^2\right)^2\left(\int_\Omega u^4\right)^2.
\end{align*}
For $\lambda=\lambda_i$, we get $L_{\lambda_i}\ge\frac{C_S^2\min\{1,\lambda_1^2,\lambda_2^2\}}{4}$, $i=1, 2$, and $L:=\min\{L_{\lambda_1},L_{\lambda_2}\}\ge\frac{C_S^2\min\{1,\lambda_1^2,\lambda_2^2\}}{4}$.
Now, let \[\varphi_\varepsilon(x)=\begin{cases}
    \varepsilon^{-\frac{N}{2}}(1-\varepsilon^{-\frac12}|x|)&\text{if $|x|\le\sqrt\varepsilon$}\\
    0&\text{if $|x|\ge\sqrt\varepsilon$}.
\end{cases}
\]
By \cite[eq. $(2.11)-(2.12)$]{lin1988large} we have that, for each $q\ge1$, there exist explicit constants $K_q,K>0$, defined in \eqref{const thm 1.1}, such that
\begin{equation}\label{stime phi lin-ni-takagi}
\int_\Omega|\varphi_\varepsilon|^q=K_q\varepsilon^{(1-q)\frac N2},\quad \int_\Omega|\nabla\varphi_{\varepsilon}|^2=K\varepsilon^{-1-\frac N2}.
\end{equation}
We test the Mountain Pass level defined in Remark \ref{Remark Mountain Pass} with $(\varphi_\varepsilon,\varphi_\varepsilon)$:
\begin{align*}
    g_\beta= \widetilde c_\beta&\le\max_{t>0}\mathcal J_\beta(t\varphi_\varepsilon,t\varphi_\varepsilon)=\frac{1}{8(1+\beta)}\frac{\left(\|\varphi_\varepsilon\|_{\lambda_1}^2+\|\varphi_\varepsilon\|_{\lambda_2}^2 \right)^2}{|\varphi_\varepsilon|_4^4}\\
    &\le\frac{1}{8(1+\beta)}\frac{(2K\varepsilon^{-1-\frac N2}+\lambda_1K_2\varepsilon^{-\frac N2}+\lambda_2K_2\varepsilon^{-\frac N2})^2}{K_4\varepsilon^{-\frac{3N}{2}}},
\end{align*}
where we are using \eqref{stime phi lin-ni-takagi}. Choosing $\varepsilon=(\lambda_1+\lambda_2)^{-1}$, we get: 
\[g_\beta\le\frac{(2K+K_2)^2}{8K_4(1+\beta)}(\lambda_1+\lambda_2)^{\frac{4-N}{2}}.\]
 If \eqref{eq least energy non constant} holds, then
 \[g_\beta<\min\{L,\mathcal J_\beta(c_1,c_2)\},\]
and the proof is complete.
\end{proof}

\begin{remark}\label{remark constants iii}
We note that \eqref{eq least energy non constant} holds if  $\lambda_1,\lambda_2>0$ satisfy
\[
\frac{(\lambda_1+\lambda_2)^{\frac{4-N}{2}}}{\lambda_1\lambda_2}<\frac{2K_4}{(2K+K_2)^2}|\Omega|,
\]
and $\beta>0$ is large enough. In turn, the inequality above holds if:
\begin{itemize}
\item either $\lambda_1$ or $\lambda_2$ is large enough when $N=3,4$;
\item both $\lambda_1$ and $\lambda_2$ are large enough when $N=1,2$.
\end{itemize}
\end{remark}

\noindent
\begin{remark}\label{rem:lambda_1=lambda_2}
In the particular case $\lambda_1=\lambda_2$, one can simplify the statement of Theorem \ref{Theorem 1.1}-$(ii)$. Indeed, in this situation, 
\[g_\beta\le\mathcal J_\beta\left(\frac{\omega}{\sqrt{1+\beta}},\frac{\omega}{\sqrt{1+\beta}}\right)=\frac{2L_{\lambda}}{1+\beta},\]
where $\omega\in H^1(\Omega)$ is the least energy solution of \eqref{scalar eq}. Then, for $\beta>1$, we have $g_\beta< L_{\lambda}$ and any ground state solution is not semi-trivial in this case. Thus, reasoning as in the proof of Theorem \ref{Theorem 1.1}-$(ii)$, there exists a least energy solution which is non-constant if $\beta>1$ and
\begin{equation*}\label{eq non costanti duale}
    \lambda^{\frac N2}>\frac{(2K+K_2)^2}{4K_4|\Omega|}.
\end{equation*}
\end{remark}

\subsection{Weakly cooperative case}\label{weakly cooperative case}

We now give the proof of Theorem \ref{Theorem 1.1}-$(i)$ and  we assume $\lambda_1,\lambda_2\in\R$ and $\beta\in(0,1)$.
We define the Nehari-type set:
\[\mathcal N_\beta:=\biggl\{(u,v)\in H\setminus \tilde H\ :\ u,v\not\equiv0,\  \langle\mathcal J_\beta'(u,v),(u,0)\rangle=0,\ \langle\mathcal J_\beta'(u,v),(0,v)\rangle=0,\ \langle\mathcal J_\beta'(u,v),\tilde{\boldsymbol w}\rangle =0\ \ \forall\ \tilde {\boldsymbol w}\in\tilde H \biggr\},\]
and 
\[
m_\beta:=\inf_{\mathcal N_\beta}\mathcal J_\beta.
\]  One can show that $\mathcal N_\beta$ is a $C^1$-submanifold with the same proof of \cite[Lemma 2.3]{xu2025least}, while the first part of  \cite[Theorem 1.1]{xu2025least} implies that $\mathcal N_\beta$ is closed with respect to the strong topology in $H$ for 
$\beta<\underline{\beta}(\lambda_1,\lambda_2)$ with $\underline{\beta}(\lambda_1,\lambda_2)$ defined in \cite[Theorem 1.1]{xu2025least}. In particular, we note that 
\begin{equation*}\label{beta segnato lambda positivi}
 \underline{\beta}(\lambda_1,\lambda_2)=\frac{\min\left\{\sqrt{L_{\lambda_1}},\sqrt{L_{\lambda_2}}\right\}}{\sqrt{L_{\lambda_1}+L_{\lambda_2}}} \quad \text{ in the particular case $\lambda_1,\lambda_2>0$. }
\end{equation*}
We emphasize that, although the proofs in \cite{xu2025least} are developed in \( H_0^1(\Omega) \) since the paper deals with Dirichlet boundary condition, they remain valid in \( H^1(\Omega) \).

The Ekeland's Variational Principle \cite{ekeland1974variational} implies that there exists a minimizing sequence $\{(u_n,v_n)\}\subset\mathcal N_\beta$ of $m_\beta$, and $\{\boldsymbol\Lambda_n\}\subset\R^2\times\tilde H$, $\boldsymbol \Lambda_n=(\Lambda_{n,1},\Lambda_{n,2},\tilde{\boldsymbol{\Lambda}}_n)=(\Lambda_{n,1},\dots,\Lambda_{n,\dim\tilde H+2})$, such that 
\begin{equation}\label{PS restrizione}
\mathcal J_\beta'(u_n,v_n)-\boldsymbol\Lambda_n\cdot\mathcal G'(u_n,v_n)=o(1)\ \text{in $H'$},\end{equation}
where $\mathcal G:H\to\R^2\times\tilde H$ is defined as follows:
\[\mathcal G(u,v):=\left(\langle\mathcal J_\beta'(u,v),(u,0)\rangle,\langle\mathcal J_\beta'(u,v),(0,v)\rangle,P_{\tilde H}\nabla\mathcal J_\beta(u,v)\right),\]
 $P_{\tilde H}:H\to\tilde H$ is the orthogonal projection, and the scalar product in $\mathbb{R}^2 \times \tilde H$ is given by
\[
(t_1, s_1, \tilde w) \cdot (t_2, s_2, \tilde z) := t_1 t_2 + s_1 s_2 + \langle \tilde w, \tilde z \rangle.
\]

Moreover, since  $\mathcal J_\beta$ is coercive on $\mathcal N_\beta$ by \cite[Lemma 2.2]{xu2025least}, we infer that $\{(u_n,v_n)\}$ is bounded in $H$.

\begin{proposition}\label{PS sequence w.c.}
Assume that $N\le4$ and $\beta\in(0,1)$ is small enough.
    Let $\{(u_n,v_n)\}\subset\mathcal N_\beta$ be a bounded sequence such that \eqref{PS restrizione} holds and $\mathcal J_\beta(u_n,v_n)=m_\beta+o(1)$. Then $\{(u_n,v_n)\}$ is a $(PS)_{m_\beta}$-sequence, i.e. $\mathcal J_\beta(u_n,v_n)=m_\beta+o(1)$ and $\mathcal J_\beta'(u_n,v_n)=o(1)$ in $H'$.
\end{proposition}
\begin{proof}
Since $\{(u_n,v_n)\}$ is bounded in $H$, we can assume that $u_n\wto u, v_n\wto v$ in $H^1(\Omega)$ and that $(u,v)\not\equiv(0,0)$. Indeed, if $(u,v)\equiv(0,0)$ the thesis follows by weak convergence.
From \cite[equations (18) and (19)]{xu2025least}, we deduce that there exists $\delta>0$ such that $|u_n|_4,|v_n|_4\ge\delta$ for every $\beta\in(0,1)$ small enough.
Notice that
\begin{align*}
\langle\mathcal G'(u_n,v_n),(tu_n+\tilde{ w}_1,s v_n+\tilde{ w}_2)\rangle\cdot(\Lambda_{n,1},\Lambda_{n,2},\tilde{\boldsymbol \Lambda}_n)&=\Lambda_{n,1}\langle\mathcal J_\beta''(u_n,v_n),((tu_n+\tilde{ w}_1,s v_n+\tilde{ w}_2),(u_n,0))\rangle\\
&\quad+\Lambda_{n,2}\langle\mathcal J_\beta''(u_n,v_n),((tu_n+\tilde{w}_1,s v_n+\tilde{ w}_2),(0,v_n))\rangle\\
&\quad+\langle\mathcal J_\beta''(u_n,v_n),((tu_n+\tilde{w}_1,s v_n+\tilde{ w}_2),\tilde{\boldsymbol \Lambda}_n)\rangle,
\end{align*}
for every $(t,s)\in\R$ and $\tilde{\boldsymbol w}\in\tilde H$. 

For every $(\varphi,\psi),(\zeta,\eta)\in H$ we have also that

\begin{align*}
\langle\mathcal J_\beta''(u,v),((\varphi,\psi),(\zeta,\eta)\rangle
&= \int_{\Omega}\nabla\zeta\cdot\nabla\varphi + \lambda_1\int_{\Omega}\zeta\varphi
+ \int_{\Omega}\nabla\eta\cdot\nabla\psi + \lambda_2\int_{\Omega}\eta\psi \\
&\quad -3\int_{\Omega}u^{2}\,\zeta\varphi -3\int_{\Omega}v^{2}\,\eta\psi
- \beta\int_{\Omega}\zeta\varphi\,v^{2} -2\beta\int_{\Omega}u\varphi\,v\eta\\
&\quad -2\beta\int_{\Omega}u\zeta\,v\psi -\beta\int_{\Omega}u^{2}\,\eta\psi.
\end{align*}

Testing \eqref{PS restrizione} with $(\Lambda_{n,1}u_n,\Lambda_{n,2}v_n)+\tilde{\boldsymbol \Lambda}_n$, we have:
\begin{align*}
o(\|{\boldsymbol \Lambda}_n\|)&=\langle\mathcal G'(u_n,v_n), (\Lambda_{n,1}u_n,\Lambda_{n,2}v_n)+\tilde{\boldsymbol \Lambda}_n \rangle \cdot (\Lambda_{n,1},\Lambda_{n,2},\tilde{\boldsymbol \Lambda}_n)\\
&=B(\tilde{\boldsymbol\Lambda}_n,\tilde{\boldsymbol\Lambda}_n)-\int_\Omega\left[u_n^2\tilde\Lambda_{n,1}^2+v_n^2\tilde\Lambda_{n,2}^2+\beta(\tilde\Lambda_{n,1}^2v_n^2+\tilde\Lambda_{n,2}^2u_n^2)\right]\\
&\quad-2\int_\Omega\big[(\Lambda_{n,1}u_n+\tilde\Lambda_{n,1})^2u_n^2+(\Lambda_{n,2}v_n+\tilde\Lambda_{n,2})^2v_n^2\\
&\quad+2\beta u_nv_n(\Lambda_{n,1}u_n+\tilde\Lambda_{n,1})(\Lambda_{n,2}v_n+\tilde\Lambda_{n,2})\big].\\
\end{align*}
In particular,
\begin{equation}\label{eq o(1)}
\begin{aligned}
o(1)&=\frac{B(\tilde{\boldsymbol\Lambda}_n,\tilde{\boldsymbol\Lambda}_n)}{\|{\boldsymbol\Lambda}_n\|}-\frac{1}{\|{\boldsymbol\Lambda}_n\|}\int_\Omega\left[u_n^2\tilde\Lambda_{n,1}^2+v_n^2\tilde\Lambda_{n,2}^2+\beta(\tilde\Lambda_{n,1}^2v_n^2+\tilde\Lambda_{n,2}^2u_n^2)\right]\\
&\quad-\frac{2}{\|{\boldsymbol\Lambda}_n\|}\int_\Omega\big[(\Lambda_{n,1}u_n+\tilde\Lambda_{n,1})^2u_n^2+(\Lambda_{n,2}v_n+\tilde\Lambda_{n,2})^2v_n^2\\
&\quad+2\beta u_nv_n(\Lambda_{n,1}u_n+\tilde\Lambda_{n,1})(\Lambda_{n,2}v_n+\tilde\Lambda_{n,2})\big].
\end{aligned}
\end{equation}
Since the matrix
\[
\begin{bmatrix}
u_n^2&\beta u_nv_n\\
\beta u_nv_n&v_n^2
\end{bmatrix}
\]
is positive (semi-)definite, we obtain that
\[
\int_\Omega\big[(\Lambda_{n,1}u_n+\tilde\Lambda_{n,1})^2u_n^2+(\Lambda_{n,2}v_n+\tilde\Lambda_{n,2})^2v_n^2+2\beta u_nv_n(\Lambda_{n,1}u_n+\tilde\Lambda_{n,1})(\Lambda_{n,2}v_n+\tilde\Lambda_{n,2})\big]\ge0.
\]

Combining this with \eqref{eq o(1)}, we deduce that $\{\|\boldsymbol\Lambda_n\|\}$ is bounded in $\R$; indeed, if this is not the case, then (up to a subsequence)
\[
o(1)\leq \frac{B(\tilde{\boldsymbol\Lambda}_n,\tilde{\boldsymbol\Lambda}_n)}{\|{\boldsymbol\Lambda}_n\|}\to -\infty,
\]
a contradiction. Hence, we can assume that $\boldsymbol\Lambda_n\to\boldsymbol\Lambda=(\Lambda_1,\Lambda_2,\tilde{\boldsymbol \Lambda})\in\R^2\times\tilde H$.  If $\tilde{\boldsymbol\Lambda}\not\equiv0$, we have that
\[
o(1)\le\frac{B(\tilde{\boldsymbol\Lambda}_n,\tilde{\boldsymbol\Lambda}_n)}{\|{\boldsymbol\Lambda}_n\|}-\frac{1}{\|{\boldsymbol\Lambda}_n\|}\int_\Omega\left[u_n^2\tilde\Lambda_{n,1}^2+v_n^2\tilde\Lambda_{n,2}^2+\beta(\tilde\Lambda_{n,1}^2v_n^2+\tilde\Lambda_{n,2}^2u_n^2)\right]
\]
and the right hand side goes to a negative number as $n\to+\infty$, a contradiction. Therefore,  $\tilde{\boldsymbol\Lambda}\equiv0$, and one has that 
\begin{align*}
o(\|\boldsymbol\Lambda_n\|)&=-2\Lambda_{n,1}^2\int_\Omega u_n^4-4\beta\Lambda_{n,1}\Lambda_{n,2}\int_\Omega u_n^2v_n^2-2\Lambda_{n,2}^2\int_\Omega v_n^4+o(\|\boldsymbol{\Lambda}_n\|)\\
&\le-c_n|(\Lambda_{n,1},\Lambda_{n,2})|^2+o(\|\boldsymbol{\Lambda}_n\|),
\end{align*}
where $c_n>0$ is the first eigenvalue of the matrix
\[
\begin{bmatrix}
-2\int_\Omega u_n^4&-2\beta\int_\Omega u_n^2v_n^2\\
-2\beta\int_\Omega u_n^2v_n^2&-2\int_\Omega v_n^4
\end{bmatrix},
\]
 and $c_n\ge\varepsilon$ for some $\varepsilon>0$, since $|u_n|_4,|v_n|_4\ge\delta$. Thus, $\Lambda_{n,1},\Lambda_{n,2}=o(1)$.
\end{proof}

\begin{theorem}\label{Compactness w.c.}
   Let $N=4$ and  assume $\beta\in(0,1)$ is small enough. Let $\{(u_n,v_n)\}\subset\mathcal N_\beta$ be a Palais-Smale sequence at level $m_\beta$ and assume that 
   \begin{equation}\label{threshold w.c.}
m_\beta\le L_{\lambda_1}+L_{\lambda_2}\quad \text{and } \quad m_\beta<\frac{S^2}{4(1+\beta)}
 \end{equation}
   holds, where $L_{\lambda_i}$ is the least energy of the scalar equation $-\Delta u+\lambda_iu=u^3$ in $H^1(\Omega)$. Then, there exists $(u,v)\in\mathcal N_\beta$ such that $\mathcal J_\beta(u,v)=m_\beta$ and $\mathcal J_\beta'(u,v)=0$.
\end{theorem}
\begin{proof}
The boundedness of $\{(u_n,v_n)\}$ follows by the coercivity of the functional \({\mathcal J_\beta}|_{\mathcal N_\beta}\). Hence, we can assume that $u_n\wto u, v_n\wto v$ weakly in \(H^1(\Omega)\). 
We denote $\theta_n:=u_n-u,\sigma_n:=v_n-v$. We can repeat the proof of Theorem \ref{PS condition beta>0} to obtain \eqref{Jb'=0 eq1}-\eqref{eq. energia}.  We may also assume that  the following limits
 are finite:
\[\xi_1:=\lim_{n\to+\infty}\int_{\Omega}|\nabla\theta_n|^2,\ \ \xi_2:=\lim_{n\to+\infty}\int_{\Omega}|\nabla\sigma_n|^2,\]
 and we can prove that
\[
\xi_1\le\frac{2}{S^2}(\xi_1^2+\beta\xi_1\xi_2),\qquad \xi_2\le\frac{2}{S^2}(\xi_2^2+\beta\xi_1\xi_2).
\]
\textbf{Step 1}: $(u,v)\not\equiv(0,0)$. Assume by contradiction that $(u,v)\equiv(0,0)$.
 From \cite[equations (18) and (19)]{xu2025least}, we deduce that there exists $\delta>0$ such that $|u_n|_4,|v_n|_4\ge\delta$ for every $\beta\in(0,1)$ small enough and  $\xi_1,\xi_2>0$. Thus, $\xi_1+\xi_2\ge\frac{S^2}{1+\beta}$. 
This contradicts $m_\beta<\frac{S^2}{4(1+\beta)}$. \\
\textbf{Step 2}:  \(u \not\equiv 0\) and \(v \not\equiv 0\). 
This part of the proof is inspired by \cite[Proof of Theorem 1.3, pages 539-540]{chen2012positive}, but we exploit the Cherrier's inequality instead of the Sobolev inequality.
Assume that $u\not\equiv0, v\equiv0$. Thus, \(u\) is a solution 
for $-\Delta u+\lambda_1u=u^3$ in $\Omega$ with $\frac{\partial u}{\partial\nu}=0$ on $\partial\Omega$ and $\mathcal J_\beta(u,0)\ge L_{\lambda_1}$.
Furthermore, 
\[\int_{\Omega}|\nabla v_n|^2 + o(1) =\int_{\Omega}v_n^4+\beta\int_{\Omega}\theta_n^2v_n^2\le\left(\frac{\sqrt2}{S}+\varepsilon\right)^2|\nabla v_n|_2^4+o(1).\]
Therefore, $\xi_2\ge\frac{S^2}{2}$  and this implies that
\[m_\beta=\mathcal J_\beta(u,0)+\frac14\xi_2\ge L_{\lambda_1}+\frac{S^2}{8}> L_{\lambda_1}+L_{\lambda_2},
\]
where we used $L_{\lambda_2}<\frac{S^2}{8}$, see \eqref{threshold scalar equation}.
We obtain a contradiction with \eqref{threshold w.c.}. 
The case $u\equiv0,v\not\equiv0$ is similar. 
Now,
\[m_\beta+o(1)=\mathcal J_\beta(u_n,v_n)=\mathcal J_\beta(u,v)+\frac14\int_\Omega(|\nabla {\theta_n}|^2+|\nabla {\sigma_n}|^2)\ge m_\beta.\]
Hence, the convergence is strong, and we obtain $\mathcal J_\beta(u,v)=m_\beta$.
\end{proof}

\begin{lemma}\label{stime energia w.c. parte1}
Let $N\le4$ and $\beta\in(0,1)$. We have that $m_\beta\leq L_{\lambda_1}+L_{\lambda_2}$.
\end{lemma}
\begin{proof}
Let $u_{i}$ be a least energy solution of \eqref{scalar eq} at level $L_{\lambda_i}$. For every $t,s>0$ and $\tilde{\boldsymbol w}\in\tilde H$,  we have
\begin{align*}
    \mathcal J_\beta(tu_1+\tilde w_1,su_2+\tilde w_2)&\le\frac{1}2|t\nabla u_1+\nabla\tilde w_1|_2^2+\frac{\lambda_1}{2}|tu_1+\tilde w_1|_2^2-\frac{1}{4} |tu_1+\tilde w_1|_4^4\\
    &\hspace{3mm}  +\frac{1}2|s\nabla u_2+\nabla\tilde w_2|_2^2+\frac{\lambda_2}{2}|su_2+\tilde w_2|_2^2-\frac{1}{4}|su_1+\tilde w_2|_4^4\\
    &\le\max_{\substack{t>0\\ \tilde w_1\in\tilde H_1}}\frac{1}2|t\nabla u_1+\nabla\tilde w_1|_2^2+\frac{\lambda_1}{2}|tu_1+\tilde w_1|_2^2-\frac{1}{4}|su_1+\tilde w_1|_4^4\\
    &\hspace{3mm} +\max_{\substack{s>0\\ \tilde w_2\in\tilde H_2}}\frac{1}2|s\nabla u_2+\nabla\tilde w_2|_2^2+\frac{\lambda_2}{2}|su_2+\tilde w_2|_2^2-\frac{1}{4} |su_1+\tilde w_2|_4^4\\
    &=L_{\lambda_1}+L_{\lambda_2},
\end{align*}
where we used the characterization of the ground state \eqref{caratterizzazione l.e.s. scalare} and the fact that $\beta>0$.  Now, we can choose  $t,s>0$ and  $\tilde{\boldsymbol w}\in\tilde H$  such that $(tu_1+\tilde w_1,su_2+\tilde w_2)\in\mathcal N_\beta$, according to \cite[Lemma 2.6]{xu2025least}. Then $m_\beta<L_{\lambda_1}+L_{\lambda_2}$.
\end{proof}

\begin{proposition}\label{stime energia w.c. parte2}
Assume $N=4$ and $\beta\in(0,1)$.
We have that the compactness threshold \eqref{threshold w.c.} is satisfied.
\end{proposition}
\begin{proof} 
The condition $m_\beta\leq L_{\lambda_1}+L_{\lambda_2}$ follows by Lemma \ref{stime energia w.c. parte1}. It remains to prove the inequality $m_\beta<\frac{S^2}{4(1+\beta)}$.
Arguing as in \cite[Lemma 2.6]{xu2025least} we find $\boldsymbol{t}_\varepsilon=(t_{\varepsilon,1},t_{\varepsilon,2})\in[0,+\infty)^2$ and $\tilde{\boldsymbol w}_\varepsilon\in\tilde H$ such that
\[
\mathcal J_\beta(\boldsymbol{t}_\varepsilon w_\varepsilon+\tilde{\boldsymbol w}_\varepsilon)=\sup_{\substack{t,s\in\R\\ \tilde{\boldsymbol w}\in\tilde H}}\mathcal J_\beta(tw_\varepsilon+\tilde w_1,sw_\varepsilon+\tilde w_2).
\]
Let $\tau_\varepsilon>0$ be such that 
\[\max_{\substack{\tau>0}}\mathcal J_\beta(\tau w_\varepsilon,\tau w_\varepsilon)=\mathcal J_\beta(\tau_\varepsilon w_\varepsilon,\tau_\varepsilon w_\varepsilon).\]
Notice that $\tau_\varepsilon>0$ is well-defined by \eqref{stima norma 2}. 
Reasoning as in the last part of the proof of \cite[Lemma 3.5]{SzulkinWethWillem} we obtain that
\[
\frac14\ \frac{B_1\left(w_\varepsilon, w_\varepsilon\right)^2}{\displaystyle\int_\Omega w_\varepsilon^4}+o(\sqrt\varepsilon)\ge\max_{\substack{\tau>0\\\tilde z\in\tilde H_1}}\mathcal J_\beta(\tau w_\varepsilon+\tilde z,0)\ge\max_{\substack{\tau>0\\\tilde w\in\tilde H}}\mathcal J_\beta(\tau w_\varepsilon+\tilde w_1,\tilde w_2).
\]
Using \eqref{stima norma 2} and $\beta\in(0,1)$, we have that, 
\begin{align*}
\mathcal J_\beta(\boldsymbol{t}_\varepsilon w_\varepsilon+\tilde{\boldsymbol w}_\varepsilon)&\ge\mathcal J_\beta(\tau_\varepsilon w_\varepsilon,\tau_\varepsilon w_\varepsilon)=\frac{\tau_\varepsilon^2}{2} B( w_\varepsilon, w_\varepsilon)-\frac{(1+\beta)\tau_\varepsilon^4}{2}\int_\Omega w_\varepsilon^4\\
&=\max_{\tau>0}\left[ 
\frac{\tau^2}{2}[B_1(w_\varepsilon,w_\varepsilon)+B_2(w_\varepsilon,w_\varepsilon)]-\frac{(1+\beta)\tau^4}{2}\int_\Omega w_\varepsilon^4\right]=\frac{[B_1(w_\varepsilon,w_\varepsilon)+B_2(w_\varepsilon,w_\varepsilon)]^2}{(8+8\beta)\displaystyle\int_\Omega w_\varepsilon^4}\\
&=\frac{4|\nabla w_\varepsilon|_2^4+o(\sqrt\varepsilon)}{8(1+\beta)|w_\varepsilon|_4^4}=\frac{B_1(w_\varepsilon,w_\varepsilon)^2}{2(1+\beta)|w_\varepsilon|_4^4}+\frac{o(\sqrt\varepsilon)}{8(1+\beta)|w_\varepsilon|_4^4}=\frac{B_1(w_\varepsilon,w_\varepsilon)^2}{2(1+\beta)|w_\varepsilon|_4^4}+o(\sqrt\varepsilon)\\
&>\frac14\ \frac{B_1\left(w_\varepsilon, w_\varepsilon\right)^2}{\displaystyle\int_\Omega w_\varepsilon^4}+o(\sqrt\varepsilon)\ge\max_{\substack{\tau>0\\\tilde w\in\tilde H}}\mathcal J_\beta(\tau w_\varepsilon+\tilde w_1,\tilde w_2),
\end{align*}
for $\varepsilon>0$ small enough. 
Then, $t_{\varepsilon,2}>0$ and one can prove $t_{\varepsilon,1}>0$ similarly. If $\boldsymbol{t}_\varepsilon\in(0,+\infty)^2$, we can estimate 
\[m_\beta\le\max_{\substack{t,s>0\\ \tilde {\boldsymbol w}\in\tilde H}}\mathcal J_\beta(tw_\varepsilon+\tilde w_1,sw_\varepsilon+\tilde w_2)<\frac{S^2}{4(1+\beta)}\]
using the same arguments of  Proposition \ref{soglia}.
Indeed,
\begin{align*}
\mathcal J_\beta(tw_\varepsilon+\tilde w_1,sw_\varepsilon+\tilde w_2)&\le c_1\bigl[t^2+s^2+t^3(\|\tilde w_1\|+\|\tilde w_2\|)+s^3(\|\tilde w_1\|+\|\tilde w_2\|)\sqrt\varepsilon\\
&\quad+t(\|\tilde w_1\|+\|\tilde w_2\|)+s(\|\tilde w_1\|+\|\tilde w_2\|)\bigr]-c_2(t^4+\|\tilde w_1\|^4+\|\tilde w_2\|^4).
\end{align*}
Thus, one can assume that $t,s,\|\tilde w\|\le C$. Moreover,

\begin{align*}
      \mathcal J_\beta(\tilde w_1+tw_\varepsilon,\tilde w_2+tw_\varepsilon)&\le \mathcal J_\beta(tw_\varepsilon,sw_\varepsilon)-t^3\int_{\Omega} (\tilde w_1+\tilde w_2)w_\varepsilon^3-s^3\int_{\Omega} (\tilde w_1+\tilde w_2)w_\varepsilon^3 \\
      &\quad-\beta t^2s\int_{\Omega}  \tilde w_2 w_\varepsilon^3-\beta ts^2\int_\Omega \tilde w_1 w_\varepsilon -2\beta ts\int_{\Omega} \tilde w_1\tilde w_2w_\varepsilon^2\\
       &\quad -k_2(\|\tilde w_1\|^4+\|\tilde w_2\|^4) -\frac{\beta k_1}{2} \|\tilde w_1\|^2 \|\tilde w_2\|^2+tB_1(\tilde w_1,w_\varepsilon)+sB_2(\tilde w_2,w_\varepsilon).
      \end{align*}
Therefore,
\begin{align*}
\mathcal J_\beta(tw_\varepsilon+\tilde w_1,sw_\varepsilon+\tilde w_2)&\le\mathcal J_\beta(tw_\varepsilon,sw_\varepsilon)+o(\sqrt\varepsilon).
\end{align*}
Since 
\begin{align*}
\max_{t,s>0}\mathcal J_\beta(tw_\varepsilon,sw_\varepsilon)&=\frac14\left[\frac{B_1(w_\varepsilon,w_\varepsilon)^2-\beta B_1(w_\varepsilon,w_\varepsilon)\cdot B_2(w_\varepsilon,w_\varepsilon)}{(1-\beta^2)|w_\varepsilon|_4^4}+\frac{B_2(w_\varepsilon,w_\varepsilon)-\beta B_1(w_\varepsilon,w_\varepsilon)\cdot B_2(w_\varepsilon,w_\varepsilon)}{(1-\beta^2)|w_\varepsilon|_4^4}\right]\\
        &=\frac14\biggl[\frac{(1-\beta)\left(\frac{S^2}{2}-C_1\sqrt\varepsilon+O(\varepsilon)+O(\varepsilon|\log\varepsilon|)\right)^2}{(1-\beta^2)\left(\frac{S^2}{2}-C_2\sqrt\varepsilon+O(\varepsilon)\right)}\\
        &+\frac{(1-\beta)\left(\frac{S^2}{2}-C_1\sqrt\varepsilon+O(\varepsilon)+O(\varepsilon|\log\varepsilon|)\right)^2}{(1-\beta^2)\left(\frac{S^2}{2}-C_2\sqrt\varepsilon+O(\varepsilon)\right)}\biggr]\\
        &=\frac{S^2}{4(1+\beta)}-C\sqrt\varepsilon+o(\sqrt\varepsilon),
\end{align*}
we obtain that
\[
m_\beta\le\mathcal J_\beta(tw_\varepsilon,sw_\varepsilon)+o(\sqrt\varepsilon)\le\frac{S^2}{4(1+\beta)}-C\sqrt\varepsilon+o(\sqrt\varepsilon)<\frac{S^2}{4(1+\beta)}
\]
for $\varepsilon>0$ small enough. This concludes the proof.
\end{proof}

\begin{proof}[Proof of Theorem \ref{Theorem 1.1}-$(i)$]\hypertarget{proof of thm1.1 i}{} 
By Ekeland's Variational Principle \cite{ekeland1974variational}, we can assume that there exists a minimizing sequence $\{(u_n,v_n)\}\subset \mathcal N_\beta$ such that \eqref{PS restrizione} holds. 
and the coercivity of $\mathcal J_\beta$ on $\mathcal N_\beta$ implies also the boundedness of the sequence. According to Proposition \ref{PS sequence w.c.}, such minimizing sequence is also a $(PS)_{m_\beta}$-sequence for $\mathcal{J}_\beta$. Now, if $N\le3$ the $(PS)$-condition follows by Theorem \ref{PS condition beta>0} and $u,v\not\equiv0$ by Lemma \ref{stime energia w.c. parte1}. On the other hand, if $N=4$, Theorem \ref{Compactness w.c.} and Proposition \ref{stime energia w.c. parte2} imply the existence of a weak solution $(u,v)\in H$ of \eqref{Pb} such that $u,v\not\equiv0$ and $\mathcal J_\beta(u,v)=m_\beta$. Since any fully nontrivial critical point of $\mathcal J_\beta$ belongs to $\mathcal N_\beta$, we have that $m_\beta=l_\beta$ and we proved the existence of a least energy solution of \eqref{Pb}.

We now prove that any least energy solution is non-constant. We focus only on the case in which constant fully nontrivial solutions $(c_1,c_2)$ exist, see Proposition \ref{constant solution}.
 In particular,  if $(c_1,c_2)$ is a least energy solution, then 
 \[m_\beta=\mathcal J_\beta(c_1,c_2)=\frac{\lambda_1^2-2\beta\lambda_1\lambda_2+\lambda_2^2}{4(1-\beta^2)}|\Omega|.\] 
 From Lemma \ref{stime energia w.c. parte1} and \eqref{stima ground state scalare}, we obtain that 
\[\frac{\lambda_1^2-2\beta\lambda_1\lambda_2+\lambda_2^2}{4(1-\beta^2)}|\Omega|=m_\beta<L_{\lambda_1}+L_{\lambda_2}\le M\left(\lambda_1^{\frac{4-N}{2}}+\lambda_2^{\frac{4-N}{2}}\right), \]
which contradicts \eqref{eq non costanti w.c.}. 
\end{proof}
\begin{remark}\label{remark constants i}
 By Young's inequality, we obtain that 
 \[\frac{\lambda_1^2-2\beta\lambda_1\lambda_2+\lambda_2^2}{4(1-\beta^2)}|\Omega|\ge\frac{\lambda_1^2+\lambda_2^2}{4(1+\beta)}|\Omega| 
 .\]
Thus, in view of Theorem \ref{Theorem 1.1}-$(i)$, there exists a non-constant least energy solution if
\[
\beta<\min\left\{\frac{\lambda_1^2+\lambda_2^2}{4M\left(\lambda_1^{\frac{4-N}{2}}+\lambda_2^{\frac{4-N}{2}}\right)}-1,\underline{\beta}(\lambda_1,\lambda_2)\right\}.
\]
\end{remark}

\section{Competitive  case for \texorpdfstring{$\lambda_1,\lambda_2>0$}{lambda1,lambda2>0}}\label{competitive lambda positive}
Assume that $\beta<0$ and that $\lambda_1,\lambda_2>0$, and prove Theorem \ref{Theorem 1.2}. In this framewok, the set $\mathcal{N}_\beta$ introduced in Subsection \ref{weakly cooperative case}  becomes
\[
\mathcal N_\beta=
\left\{
    (u,v)\in H\ :\ \ u\not\equiv0,\ v\not\equiv0,\ \|u\|_{\lambda_1}^2=\int_{\Omega}(u^4+\beta u^2v^2),\ \ 
    \|v\|_{\lambda_2}^2=\int_{\Omega}(v^4+\beta u^2v^2)
\right\}.
\]

Recall the notation:
\[m_\beta:=\inf_{(u,v)\in\mathcal N_\beta}\mathcal J_\beta(u,v).\]
Since $\mathcal J_\beta(u,v)=\frac14(\|u\|_{\lambda_1}^2+\|v\|_{\lambda_2}^2)$ for every $(u,v)\in\mathcal N_\beta$, we have that $\mathcal J_\beta$ is coercive on $\mathcal N_\beta$.
\begin{remark}\label{Nehari property 1}
   Let $(u,v)\in\mathcal N_\beta$. By the Sobolev embedding $H^1(\Omega)\hookrightarrow L^4(\Omega)$ we have that
    \begin{equation}\label{inequality 1}
    \begin{split}
        \|u\|_{\lambda_1}^2&\le \|u\|_{\lambda_1}^2-\beta\int_{\Omega}u^2v^2=\int_{\Omega}u^4\le C_{S,\lambda_1}^2\|u\|_{\lambda_1}^4\\
     \|v\|_{\lambda_2}^2&\le \|v\|_{\lambda_2}^2-\beta\int_{\Omega}u^2v^2=\int_{\Omega}v^4\le C_{S,\lambda_2}^2\|v\|_{\lambda_2}^4
    \end{split}
    \end{equation}
    Then $\|u\|_{\lambda_1},\|v\|_{\lambda_2}\ge C>0,$ where $C=\min\{C_{S,\lambda_1}^{-1},C_{S,\lambda_2}^{-1}\}$ and \[C_{S,\lambda_i}:=\inf_{u\in H^1(\Omega)\setminus\{0\}}\frac{|\nabla u|_2^2+\lambda_i|u|_2^2}{|u|_4^{2}}.\] \\ Moreover, by \eqref{inequality 1} we get that $|u|_4^4, |v|_4^4\ge C^2>0$.  In particular $m_\beta>0$.
\end{remark}

\begin{lemma}[{\cite[Lemma 2.4]{tavares2012existence}}]\label{Nehari Manifold}
Let $N\le4$. The set $\mathcal N_\beta\subset H$ is a submanifold of codimension two for any $\beta<0$ and it  is a natural constraint, i.e. every critical point $(u,v)\in\mathcal N_\beta$ for ${\mathcal J_\beta}|_{\mathcal N_\beta}$ satisfies $\mathcal J_\beta'(u,v)=0$.
\end{lemma}

\begin{lemma}\label{Lemma 1.14}
Let $N\le4$ and let $\{(u_n,v_n)\}\subset\mathcal N_\beta$ be a Palais-Smale sequence of ${\mathcal{J}_\beta}_{|_{\mathcal{N}_\beta}}$ at level $m_\beta$. Then $\{(u_n,v_n)\}$ is a bounded $(PS)_{m_\beta}$-sequence for $\mathcal J_\beta$.
\end{lemma}
\begin{proof}
The definition of Palais-Smale at level $m_\beta$ implies that $\{(u_n,v_n)\}$ is a minimizing sequence and the boundedness follows by the coercivity of ${\mathcal{J}_\beta}_{|_{\mathcal{N}_\beta}}$.
    Since  $\{(u_n,v_n)\}$ is a Palais-Smale sequence in $\mathcal N_\beta$, we have that there exist $\{\Lambda_{1,n}\}, \{\Lambda_{2,n}\}\subset\R$ such that
    \begin{equation}\label{PS-N}
    \mathcal J_\beta'(u_n,v_n)-\Lambda_{1,n}\mathcal G_1'(u_n,v_n)-\Lambda_{2,n}\mathcal G_2'(u_n,v_n)=o(1)\ \ in\ H'.
    \end{equation}
    We want to prove that $\Lambda_{1,n},\Lambda_{2,n}=o(1)$.
    We define the matrix
    \begin{equation*}\label{matrix T}
      \mathcal T(u_n,v_n):=
    \begin{bmatrix}
        &\left\langle\mathcal G_1'(u_n,v_n), (u_n,0)\right\rangle &\left\langle\mathcal G_1'(u_n,v_n), (0,v_n)\right\rangle\\
        & &\\
        &\left\langle\mathcal G_2'(u_n,v_n), (u_n,0)\right\rangle &\left\langle\mathcal G_2'(u_n,v_n), (0,v_n)\right\rangle
    \end{bmatrix}
    =\begin{bmatrix}
        &-2\displaystyle{\int_{\Omega}u_n^4} &-2\beta\displaystyle{\int_{\Omega}u_n^2v_n^2}\\
        & &\\
        &-2\beta\displaystyle{\int_{\Omega}u_n^2v_n^2}&-2\displaystyle{\int_{\Omega}v_n^4}
    \end{bmatrix}.
    \end{equation*}
    
Notice that $\mathcal T(u_n,v_n)$satisfies
    \begin{align*}
        \det\mathcal T(u_n,v_n)&\ge4\|u_n\|_{\lambda_1}^2\cdot\|v_n\|_{\lambda_2}^2>0,
    \end{align*}
     and this bound is uniform in $n$ by Remark \ref{Nehari property 1}.
  Hence, there exists a uniform constant $c>0$ such that 
  \begin{equation}\label{def neg}
      \mathcal T(u_n,v_n)
      \begin{bmatrix}
          \Lambda_{1,n}\\
          \Lambda_{2,n}
      \end{bmatrix}
      \cdot
      \begin{bmatrix}
          \Lambda_{1,n}\\
          \Lambda_{2,n}
      \end{bmatrix}
      \le-c|(\Lambda_{1,n},\Lambda_{2,n})|^2.
  \end{equation} 
    By \eqref{PS-N} we get that
    \[
    \boldsymbol{o(1)}=\mathcal T(u_n,v_n)
    \begin{bmatrix}
        \Lambda_{1,n}\\
        \Lambda_{2,n}
    \end{bmatrix}.
    \]
    Exploiting equation \eqref{def neg}, we obtain \[o(|\Lambda_{1,n},\Lambda_{2,n})|)\le-c|(\Lambda_{1,n},\Lambda_{2,n})|^2.\] Thus,
    \[|(\Lambda_{1,n},\Lambda_{2,n})|\le o(1)\implies\Lambda_{1,n},\Lambda_{2,n}=o(1).\]
     We stress the fact that the boundedness of $\{(u_n,v_n)\}$ in $H$   implies the boundedness of $\mathcal G_1',\mathcal G_2'$.  Therefore, $\mathcal J_\beta'(u_n,v_n)=o(1)$ in $H'$.
\end{proof}

We are ready to discuss the $(PS)$-condition. If $N \leq 3$, it is immediate
 by reasoning as in Theorem \ref{PS condition beta>0}. Now, we focus on the case $N = 4$.

\begin{theorem}\label{Compactness beta<0}
   Let $N=4$, $\lambda_1,\lambda_2>0$. Let $\{(u_n,v_n)\}\subset\mathcal N_\beta$ be a bounded Palais-Smale sequence at level $m_\beta$ and assume that
   \begin{equation}\label{threshold}
m_\beta <     \min\left\{L_{\lambda_1} + \dfrac{S^2}{8}, L_{\lambda_2} + \dfrac{S^2}{8}\right\} 
\end{equation}
    holds, where $L_{\lambda_i}$ is the least energy of the scalar equation $-\Delta u+\lambda_iu=u^3$ in $H^1(\Omega)$. Then, there exists $(u,v)\in\mathcal N_\beta$ such that $\mathcal J_\beta(u,v)=m_\beta$ and $\mathcal J_\beta'(u,v)=0$.
\end{theorem}
\begin{proof}
Since \(\{(u_n,v_n)\}\) is bounded, we can assume that \(u_n \rightharpoonup u\), \(v_n \rightharpoonup v\) in \(H^1(\Omega)\), with strong convergence in \(L^q(\Omega)\) for \(q \in [1,4)\) and weak convergence for \(q = 4\).
We denote $\theta_n:=u_n-u,\sigma_n:=v_n-v$ and we can repeat the proof of Theorem \ref{PS condition beta>0} to obtain \eqref{Jb'=0 eq1}-\eqref{eq. energia}. Passing to a subsequence, we may also assume that  the following limits are finite:
\[
\xi_1 := \lim_{n \to +\infty} \int_{\Omega} |\nabla \theta_n|^2, \quad \xi_2 := \lim_{n \to +\infty} \int_{\Omega} |\nabla \sigma_n|^2.
\]  
Finally, we are ready to prove that \(u \not\equiv 0\) and \(v \not\equiv 0\). This part of the proof is inspired by \cite[Proof of Theorem 1.3, pages 539-540]{chen2012positive}, but we exploit Cherrier's inequality instead of the Sobolev inequality.\\
\textbf{Step 1}. Assume that $(u,v)=(0,0)$.  Then $\theta_n=u_n, \sigma_n=v_n$. Since  $(u_n,v_n)\in\mathcal N_\beta$, $\beta<0$ and by Cherrier's inequality, see Lemma \ref{Cherrier},  we have that
\[|\nabla u_n|_2+o(1)\le|u_n|_4^2\le\left(\frac{\sqrt2}{S}+\varepsilon\right)|\nabla u_n|_2^2+o(1),\quad |\nabla v_n|_2+o(1)\le|v_n|_4^2\le\left(\frac{\sqrt2}{S}+\varepsilon\right)|\nabla v_n|_2^2+o(1).\]
Hence, $\xi_1+\xi_2\ge S^2$. Thus, $4m_\beta\ge\xi_1+\xi_2\ge S^2$ 
and this is a contradiction with  \eqref{threshold}. Indeed, since $L_{\lambda_i}<\frac{S^2}{8}$, see section \eqref{preliminaries}, we have $m_\beta<\frac{S^2}{4}$ and $4m_\beta< S^2$.\\
\textbf{Step 2}. Assume that $u\not\equiv0, v\equiv0$. Thus, $u$ is a solution for $-\Delta u+\lambda_1u=u^3$ in $\Omega$ with $\frac{\partial u}{\partial\nu}=0$ on $\partial\Omega$ and $\mathcal J_\beta(u,0)\ge L_{\lambda_1}$.
Furthermore, 
\[\int_{\Omega}|\nabla v_n|^2 + o(1) =\int_{\Omega}v_n^4+\beta\int_{\Omega}\theta_n^2v_n^2\le\left(\frac{\sqrt2}{S}+\varepsilon\right)^2|\nabla v_n|_2^4+o(1).\]
Therefore, $\xi_2\ge\frac{S^2}{2}$ and this implies that
\[m_\beta=\mathcal J_\beta(u,0)+\frac14\xi_2\ge L_{\lambda_1}+\frac{S^2}{8}.
\]
 We obtain a contradiction with \eqref{threshold}. The case $u\equiv0,v\not\equiv0$ is similar. Since $(u,v)\in\mathcal N_\beta$, by \eqref{eq. energia} we infer that
 \begin{align*}
     m_\beta+o(1)&=\mathcal J_\beta(u_n,v_n)\ge m_\beta+\frac14\int_\Omega(|\nabla \theta_n|^2+|\nabla \sigma_n|^2)+o(1).
 \end{align*}
 Then, the convergence is strong and $\mathcal J_\beta(u,v)=m_\beta$.
\end{proof}

We want to show that \eqref{threshold} holds. Let $u_1\in H^1(\Omega)\cap L^\infty(\Omega)$ be a  solution of $-\Delta u+\lambda_1u=u^3$ in $H^1(\Omega)$ and let $w_\varepsilon=U_{\varepsilon,0}(|x|)\eta(|x|)$ where $U_{\varepsilon,0}=\frac{(8\varepsilon)^{1/2}}{\varepsilon+|x|^2}$. In the Dirichlet case, the boundary condition \( u \equiv 0 \) on \( \partial\Omega \) allows us to select \(\text{supp}(\eta)\Subset \Omega\) (close to the boundary) in such a way that \( |u_1|_\infty \) remains sufficiently small; see \cite[pages 532-534]{chen2012positive}.  This strategy does not work for the Neumann case, as $u_1$ is not zero on $\partial \Omega$; moreover, we have to concentrate $U_{\varepsilon,0}$ at a boundary point with positive curvature, and do not have enough freedom to move the support of $\eta$ to the interior of $\Omega$.
\begin{proposition}
Let $N=4$. The condition \eqref{threshold} holds.
\end{proposition}
\begin{proof}
Let $u_1$ be a  least energy solution of $-\Delta u+\lambda_1u=u^3$ in $H^1(\Omega)$  and let $w_\varepsilon=U_{\varepsilon,0}(|x|)\eta(|x|)$.  We observe that Lemma \ref{Estimates for C} yields the estimates  
\[
|w_\varepsilon|_2^2 = O(\varepsilon|\log\varepsilon|),\qquad
|w_\varepsilon|_4^4 = \frac{S^2}{2} - C_2\sqrt\varepsilon + O(\varepsilon).
\]
This, in turn, implies that
    \begin{align*}
        |u_1|_4^4\cdot| w_\varepsilon|_4^4-\beta^2|u_1 w_\varepsilon|_2^4&\ge4L_{\lambda_1}\cdot| w_\varepsilon|_4^4-\beta^2|u_1^2|_{\infty}^2\cdot| w_\varepsilon|_2^4>0
    \end{align*}
    for $\varepsilon>0$ small enough. Then
\begin{equation}\label{eq t}
t^2 = \frac{\|u_1\|_{\lambda_1}^2 \cdot |w_\varepsilon|_4^4 - \beta |u_1w_\varepsilon |_2^2 \cdot \|w_\varepsilon\|_{\lambda_2}^2}{|u_1|_4^4 \cdot |w_\varepsilon|_4^4 - \beta^2 |u_1w_\varepsilon|_2^4},\quad s^2 = \frac{|u_1|_4^4 \cdot \|w_\varepsilon\|_{\lambda_2}^2 - \beta \|u_1\|_{\lambda_1}^2 \cdot |u_1w_\varepsilon|_2^2}{|u_1|_4^4 \cdot |w_\varepsilon|_4^4 - \beta^2 |u_1w_\varepsilon|_2^4}
\end{equation}  
are well-defined, and  $(tu_1,s w_\varepsilon)\in\mathcal N_\beta$.

    We compute
     \begin{align*}
        \mathcal J_\beta(tu_1,s w_\varepsilon)&=\frac14\left[\frac{4L_{\lambda_1}| w_\varepsilon|_4^4-\beta|u_1 w_\varepsilon|_2^2\cdot\| w_\varepsilon\|_{\lambda_2}^2}{4L_{\lambda_1}| w_\varepsilon|_4^4-\beta^2|u_1 w_\varepsilon|_2^4}4L_{\lambda_1}+\frac{4L_{\lambda_1}\| w_\varepsilon\|_{\lambda_2}^2-\beta|u_1 w_\varepsilon|_2^2\cdot4L_{\lambda_1}}{4L_{\lambda_1}| w_\varepsilon|_4^4-\beta^2|u_1 w_\varepsilon|_2^4}\| w_\varepsilon\|_{\lambda_2}^2\right]\\
        &=\frac14\biggl[\frac{4L_{\lambda_1}\left(\frac{S^2}{2}-C_2\sqrt\varepsilon+O(\varepsilon)\right)+O(\varepsilon|\log\varepsilon|)}{4L_{\lambda_1}\left(\frac{S^2}{2}-C_2\sqrt\varepsilon+O(\varepsilon)\right)+O(\varepsilon^2\log^2\varepsilon)}4L_{\lambda_1}\\
        &\quad+\frac{4L_{\lambda_1}\left(\frac{S^2}{2}-C_1\sqrt\varepsilon+O(\varepsilon|\log\varepsilon|)\right)+O(\varepsilon|\log\varepsilon|)}{4L_{\lambda_1}\left(\frac{S^2}{2}-C_2\sqrt\varepsilon+O(\varepsilon)\right)-O(\varepsilon^2\log^2\varepsilon)}\left(\frac{S^2}{2}-C_1\sqrt\varepsilon+O(\varepsilon)+O(\varepsilon|\log\varepsilon|)\right)\biggr]\\
        &=\frac14\biggl[ 4L_{\lambda_1}\left(\frac{S^2}{2}-C_2\sqrt\varepsilon+o(\sqrt\varepsilon)\right)\left(\frac{2}{S^2}+\frac{4}{S^4}C_2\sqrt\varepsilon+o(\sqrt\varepsilon)\right)\\
        &\quad+\left(\frac{S^2}{2}-C_1\sqrt\varepsilon+o(\sqrt\varepsilon)\right)
        \left(\frac{2}{S^2}+\frac{4}{S^4}C_2\sqrt\varepsilon+o(\sqrt\varepsilon)\right)\left(\frac{S^2}{2}-C_1\sqrt\varepsilon+o(\sqrt\varepsilon)\right)
        \biggr]\\
        &=\frac14\left[4L_{\lambda_1}(1+o(\sqrt\varepsilon))+\left(1-\frac{2}{S^2}(C_2-C_1)\sqrt\varepsilon+o(\sqrt\varepsilon)\right)\left(\frac{S^2}{2}-C_1\sqrt\varepsilon+o(\sqrt\varepsilon)\right)\right]\\
        &=\frac14\left[4L_{\lambda_1}+\frac{S^2}{2}-C\sqrt\varepsilon+o(\sqrt\varepsilon)\right],
    \end{align*}
for some $C>0$.
    Thus, $\mathcal J_\beta(tw_\varepsilon,su_1)<L_{\lambda_1}+\frac{S^2}{8}$, for $\varepsilon>0$ small enough.
     In a similar way, we get $\mathcal J_\beta(tu_2,sw_\varepsilon)<L_{\lambda_2}+\frac{S^2}{8}$ and
     \[m_\beta<\min\left\{L_{\lambda_1}+\frac{S^2}{8},L_{\lambda_2}+\frac{S^2}{8}\right\}.\qedhere\]  
\end{proof}

\begin{proof}[Proof of Theorem \ref{Theorem 1.2}]
    Let \( \{(u_n, v_n)\} \subset \mathcal{N}_\beta \) be a minimizing sequence for \( m_\beta \), i.e., \( \mathcal{J}_\beta(u_n, v_n) = m_\beta + o(1) \). Then \( \{(|u_n|, |v_n|)\} \) is a minimizing sequence for \( m_\beta \) by the definitions of \( m_\beta \) and \( \mathcal{N}_\beta \). By Ekeland's Variational Principle, see \cite{ekeland1974variational}, we can assume that \( \{(|u_n|, |v_n|)\} \) is a Palais-Smale sequence at level \( m_\beta \) for \( {\mathcal J_\beta}|_{\mathcal N_\beta} \). If \( N = 4 \), by Theorem \ref{Compactness beta<0} and Proposition \ref{threshold} there exists \( (u, v) \) such that \( \mathcal{J}_\beta(u, v) = m_\beta \), \( \mathcal{J}_\beta'(u, v) = 0 \), and \( u, v \not\equiv 0 \). Since any critical point of $\mathcal J_\beta$ belongs to $\mathcal N_\beta$, we have that $m_\beta=l_\beta$ and $(u,v)$ is a least energy solution of \eqref{Pb}.  The Strong Maximum Principle implies that \( u, v > 0 \). The same conclusion holds for \( N \leq 3 \) by  \eqref{compactness subcritical}. 
    
    It remains to prove that  the least energy solutions are not constant under the stated assumptions. Indeed, if the solution is a constant $(c_1,c_2)\in\R^2$, then it has the shape \eqref{costanti non banali}. This is not well defined when $\beta\leq -1$; if instead $\beta\in (-1,0)$, then
    \[m_\beta=\mathcal J_\beta(c_1,c_2)=\frac{\lambda_1^2+\lambda_2^2-2\beta\lambda_1\lambda_2}{4(1-\beta^2)}.\]
Let $u_i\in H^1(\Omega)$ be a least energy solution of $-\Delta u_i+\lambda_iu_i=u_i^3$ in $\Omega$, for $i=1,2$.
According to \eqref{stima ground state scalare},
we have 
\begin{align*}
m_\beta&\le\mathcal J_\beta(tu_1,su_2)=
\frac14\left[\frac{4 L_{\lambda_1}4L_{\lambda_2}-\beta 4L_{\lambda_2}|u_1u_2|_2^2}{4L_{\lambda_1}4 L_{\lambda_2}-\beta^2|u_1u_2|_2^4}4L_{\lambda_1}+\frac{4 L_{\lambda_1}4L_{\lambda_2}-\beta 4L_{\lambda_1}|u_1u_2|_2^2}{4L_{\lambda_1}4 L_{\lambda_2}-\beta^2|u_1u_2|_2^4}4L_{\lambda_2} \right]\\
&\le\frac14\left[\frac{16 L_{\lambda_1}L_{\lambda_2}-4\beta L_{\lambda_2}\sqrt{16 L_{\lambda_1}L_{\lambda_2}}}{16L_{\lambda_1} L_{\lambda_2}-16\beta^2 L_{\lambda_1}L_{\lambda_2}}4L_{\lambda_1}+\frac{16 L_{\lambda_1}L_{\lambda_2}-4\beta L_{\lambda_1}\sqrt{16 L_{\lambda_1}L_{\lambda_2}}}{16L_{\lambda_1} L_{\lambda_2}-16\beta^2 L_{\lambda_1}L_{\lambda_2}}4L_{\lambda_2}\right]\\
&= \frac{L_{\lambda_1}+L_{\lambda_2}-2\beta\sqrt{L_{\lambda_1}L_{\lambda_2}}}{4(1-\beta^2)}\le  M\frac{\lambda_1^{\frac{4-N}{2}}+\lambda_2^{\frac{4-N}{2}}-2\beta\lambda_1^{\frac{4-N}{4}}\lambda_2^{\frac{4-N}{4}}}{4(1-\beta^2)}, 
\end{align*}
where $t,s$  are as in \eqref{eq t} (with $u_2$ instead of $w_\varepsilon$) and are well-defined by H\"older's inequality and the fact that $\beta\in(-1,0)$.

Thus,
\[M\frac{\lambda_1^{\frac{4-N}{2}}+\lambda_2^{\frac{4-N}{2}}-2\beta\lambda_1^{\frac{4-N}{4}}\lambda_2^{\frac{4-N}{4}}}{4(1-\beta^2)}\ge\frac{\lambda_1^2+\lambda_2^2+2|\beta|\lambda_1\lambda_2}{4(1-\beta^2)},\]
and we get a contradiction with the hypothesis.
\end{proof}
\section{Strongly competitive case for \texorpdfstring{$\lambda_1>0$}{lambda1>0} and \texorpdfstring{$\lambda_2=0$}{lambda2>0}}\label{competitive new}
Here, we treat the case  $\lambda:=\lambda_1>0$, $\lambda_2=0=\mu_1^N(\Omega)$, and $\beta<0$ with $|\beta|$ large enough, proving Theorem \ref{main result 4}. 

Observe that $\tilde H_1=\{0\}$, while $H_2^-=\{0\}$ and $H_2^0=\text{span} \{1\}$. If we integrate the second equation of \eqref{Pb}, we obtain
\[
\int_{\Omega}v^3+\beta\int_\Omega u^2v=0\iff
\langle \mathcal{J}'_\beta(u,v),(0,1)\rangle=0.
\]
Then weak solutions of \eqref{Pb} belong to the set
\[Y_\beta=\left\{(u,v)\in\ H\ :\ \int_{\Omega}v^3+\beta\int_\Omega u^2v=0\right\}\]
and the Nehari-type set introduced in in Sections \ref{weakly cooperative case} corresponds to
\[\tildeNb=\left\{(u,v)\in Y_\beta\ :\ u\not\equiv0,\ v\not\equiv 0,\ \|u\|_{\lambda}^2=\int_{\Omega}u^4+\beta \int_\Omega u^2v^2,\  \int_{\Omega}|\nabla v|^2=\int_{\Omega}v^4+\beta \int_\Omega u^2v^2\right\}.\]
As in Section \ref{weakly cooperative case} (for the weakly cooperative case) and Section \ref{competitive lambda positive} (competitive case with $\lambda_1,\lambda_2>0$), we work with the level $m_\beta:=\inf_{\tildeNb}\mathcal J_\beta$ (although the computations in the setting of this section are different).

First of all, we see that:
\begin{equation}\label{stima basso}
    \|u\|_{\lambda}^2\le\|u\|_{\lambda}^2-\beta \int_{\Omega}u^2v^2=\int_{\Omega}u^4\le C_{S,\lambda}^{-2}\|u\|_{\lambda}^4,\quad \forall\ (u,v)\in\mathcal N_\beta.
\end{equation}
 Then, $\|u\|_{\lambda}\ge C_1>0$, where $C_1=C_{S,\lambda}$ and $C_{S,\lambda}$ is the best constant of the embedding $(H^1,\|\cdot\|_\lambda)\hookrightarrow L^4$. In particular, $m_\beta>0$.

\begin{lemma}\label{coercività sezione 5}
Let $N\le4$. Then 
\[
 \mathcal J_\beta(u,v)=\frac14(\|u\|_\lambda^2+|\nabla v|_2^2) \quad \text{on $\mathcal N_\beta$},
\]
and there exists $\kappa_1=\kappa_1(\lambda,\beta),\kappa_2=\kappa_2(\lambda)$ such that
\[
\kappa_1\|(u,v)\|_H^2 \le  \|u\|_\lambda^2+|\nabla v|_2^2
    \le \kappa_2 \|(u,v)\|_H^2  \text{ for every } (u,v)\in Y_\beta.
\]
In particular, ${\mathcal J_\beta}|_{\mathcal N_\beta}$ is coercive.
\end{lemma}
\begin{proof}
We divide the proof in two steps.\\
\textbf{Step 1}:
For any $(u,v)\in Y_\beta$ there exists $C_\beta>0$ such that
\[|u|_2^2+|v|_2^2\le C_\beta(\|u\|_\lambda^2+|\nabla v|_2^2).\]
We consider the minimizing problem 
\[C_\beta^{-1}:=\inf_{(u,v)\in Y_\beta\setminus\{(0,0)\}}\frac{\|u\|_\lambda^2+|\nabla v|_2^2}{|u|_2^2+|v|_2^2}.\]
Let $\{(u_n,v_n)\}$ be a minimizing sequence such that $|(u_n,v_n)|_2=1$.\\ We have that $(u_n,v_n)\wto (u,v)$ in $H$ and $(u,v)\in Y_\beta$, by the strong convergence in $L^p(\Omega)$ for $p\in[1,4)$. Assume that $(u,v)=(0,c)$, for some $c\in\R$. Then, $c=0$, since $(u,v)\in Y_\beta$, and \[1=\int_\Omega(u^2+v^2)=c^2|\Omega|=0, \] a contradiction. Thus $C_\beta^{-1}>0$ and the assertion follows by definition of infimum.\\
\textbf{Step 2}: Coercivity.
 By  the previous step, for every $(u,v)\in\tildeNb$ we have the following estimate:
    \begin{equation*}\label{norma equivalente}
        \begin{aligned}
             \|u\|_\lambda^2+|\nabla v|_2^2&\le \max\{1,\lambda\}\|(u,v)\|_H^2\\
    &\le\max\{1,\lambda\}(1+C_\beta) (\|u\|_\lambda^2+|\nabla v|_2^2).
        \end{aligned}
    \end{equation*}
    Notice that, for every $(u,v)\in\mathcal N_\beta$,
    \[
    \mathcal J_\beta(u,v)=\frac14(\|u\|_\lambda^2+|\nabla v|_2^2)\ge\frac{\kappa_1}{4}\|(u,v)\|^2,
    \]
the coercivity follows.
\end{proof}

\begin{proposition}\label{chiusura sezione 5}
 Let $N\le4$.   The Nehari $\tildeNb$ is a closed set with respect to the strong topology of $H$.
\end{proposition}

\begin{proof}
Let $\{(u_n,v_n)\}\subset\mathcal N_\beta$ such that $u_n\to u$ and $v_n\to v$ in $H^1(\Omega)$.
   From \eqref{stima basso} we have that $u\not\equiv0$. Now, assume that $v_n\to0$. Let us define  $w_n=v_n/\|v_n\|$, which is weakly convergent to some $w$. Since 
\[
\int_{\Omega}(v_n^3+\beta u_n^2v_n)=0,\ \text{ we obtain that } a_n:=\int_{\Omega}w_n^3+\frac{\beta}{\|v_n\|^2}\int_{\Omega}u_n^2w_n=0.
\]
If $w\not\equiv0$, we get a contradiction passing to the limit as $n\to+\infty$. Otherwise, if $w_n\wto 0$, dividing by $\|v_n\|^2$ the identity:
\[
\int_\Omega|\nabla v_n|^2=\int_\Omega v_n^4+\beta\int_\Omega u_n^2v_n^2,
\]
 we obtain:
\[\int_\Omega|\nabla w_n|^2=\int_\Omega\frac{v_n^4}{\|v_n\|^2}+\beta\int_\Omega u_n^2w_n^2\le C\|v_n\|^2+o(1).\]
Then $|\nabla w_n|_2=o(1)$, so $w_n\to0$, and this is a contradiction with $\|w_n\|=1$.
\end{proof}
We establish useful upper and lower estimates for the $L^2$-norms $|u|_2,|\nabla v|_2$ when $(u,v)$ belongs to a sublevel of the Nehari set.
\begin{lemma}\label{stime norme due}
Let $N\le3$, and let $\kappa>0$ be a positive number independent of $\beta$. For every $\beta<0$ with $|\beta|$ large enough, there exists $\delta>0$ independent of $\beta$ such that \[|u|_2,|\nabla v|_2\ge\delta,\ \ \text{$\forall$ $(u,v)\in\tildeNb\cap\{\mathcal J_\beta\le\kappa\}.$}\]
\end{lemma}
\begin{proof} By Lemma \ref{coercività sezione 5}, there exists $C=C(\lambda,\kappa)$ such that
\begin{equation}\label{eq:unif_bound_stime norme due}
\|(u,v)\|\leq C \text{ for every } (u,v)\in\tildeNb\cap\{\mathcal J_\beta\le\kappa\}.
\end{equation}
Let $\{(u_n,v_n)\}\subset\mathcal N_\beta\cap\{\mathcal J_\beta\le\kappa\}$. Then, up to a subsequence we have

$u_n\wto u$ and $v_n\wto v$ weakly in $H^1(\Omega)$.\\
\textbf{Step 1}.  
    Since $\|u_n\|_\lambda\ge C_1>0$ by \eqref{stima basso},  we have:
\begin{equation}\label{stime basso norma 4}
    \int_\Omega u_n^4\ge\|u_n\|_\lambda^2\ge C_1>0.
\end{equation}
  
    Hence, $u\not\equiv0$ and $u_n\not\to 0$ in $L^2(\Omega)$. Additionally, since $\mathcal J_\beta\le\kappa$ and $N\le3$, then $|u|_{2^*}$ is bounded from above uniformly in $\beta$ and
        \[\int_\Omega u^4\le|u|_2^r\cdot|u|_{2^*}^{4-r},\ \ r=N\left(1-\frac{4}{2^*}\right).\]
    Hence, the $L^2$-norm is bounded from below uniformly in $\beta$, by using \eqref{stime basso norma 4}.\\
\textbf{Step 2}. Assume that $|\nabla v_n|\to0$ in $L^2(\Omega)$.
We already know that $u_n\to u$ in $L^2(\Omega), L^3(\Omega)$ and $u\not\equiv0$. In particular $v_n\wto c\in\R$ 
    in $H^1$ and $(u,c)\in Y_\beta$:
    \[c^2|\Omega|+\beta \int_\Omega u^2=0\iff|\beta|=\frac{c^2|\Omega|}{|u|_2^2}.\]
    From the uniform bound of $|u|_2^2$ (and the uniform upper bound \eqref{eq:unif_bound_stime norme due}, which yields $c\leq C$), we get a contradiction for $|\beta|$ large enough.   
\end{proof}

\begin{proposition}\label{punti regolari}
Let $N\le 3$ and $\kappa>0$. The set $\tildeNb$  is a regular $C^1-$submanifold of codimension three at each point $(u,v)\in\tildeNb\cap\{\mathcal J_\beta\le\kappa\}$  for $|\beta|$ large enough.

\end{proposition}
\begin{proof}
 We see that $\tildeNb=\mathcal G^{-1}(0,0,0)$, where $\mathcal G=(\mathcal G_1,\mathcal G_2,\mathcal G_3)$ and 
 \[
 \mathcal G_1(u,v)=\|u\|_{\lambda}^2-\int_{\Omega}u^4-\beta\int_{\Omega}u^2v^2, \quad \mathcal G_2(u,v)=|\nabla v|_2^2-\int_{\Omega}v^4-\beta\int_{\Omega}u^2v^2,\quad \mathcal G_3(u,v)=\int_{\Omega}v^3+\beta\int_{\Omega}u^2v.\]
     We want to prove that $T=(T_1,T_2,T_3):=\mathcal G'(u,v)|_{\mathcal Z}:\mathcal Z\to\R^3$ is surjective with $\mathcal Z:=span\{(u,0),(0,v),(0,1)\}$. Consider the matrix:
    
\[
\mathcal T(u,v)=
\begin{bmatrix}
    &T_1(u,0)&T_1(0,v)&T_1(0,1)\\
    &T_2(u,0)&T_2(0,v)&T_2(0,1)\\
    &T_3(u,0)&T_3(0,v)&T_3(0,1)
\end{bmatrix}=
\begin{bmatrix}
    &-2\int_{\Omega}u^4 &-2\beta\int_{\Omega}u^2v^2 &2\int_{\Omega}v^3\\
    &-2\beta\int_{\Omega}u^2v^2 &-2\int_{\Omega}v^4&-2\int_{\Omega}v^3\\
    &-2\int_{\Omega}v^3&2\int_{\Omega}v^3&3\int_{\Omega}v^2+\beta\int_{\Omega}u^2
\end{bmatrix}.
\]
Computing the jacobian with the Laplace's expansion on the third column, and using $\mathcal{G}(u,v)=(0,0,0)$:
\begin{align*}
    \det(\mathcal T(u,v))&=-4\left(\int_{\Omega}v^3\right)^2\cdot2|\nabla v|_2^2-4\left(\int_{\Omega} v^3\right)^2\cdot2\|u\|_{\lambda}^2\\
    &\quad+(3|v|_2^2+\beta|u|_2^2)\det\underbrace{
    \begin{bmatrix}
    &-2\int_{\Omega}u^4 &-2\beta\int_{\Omega}u^2v^2\\
    &-2\beta\int_{\Omega}u^2v^2 &-2\int_{\Omega}v^4
    \end{bmatrix}}_{A}.
\end{align*}
According to Lemma \ref{stime norme due}, there exists $\delta>0$ such that $|u|_2,|\nabla v|_2\ge\delta$ and we have
  \begin{align*}
        \det A&=4\left[\left(\int_{\Omega}u^4\right)\left(\int_{\Omega}v^4\right)-\beta^2\left(\int_{\Omega}u^2v^2\right)^2\right]\\
        &=4\left[\left(-\beta\int_{\Omega}u^2v^2+\|u\|_{\lambda}^2\right)
        \left(-\beta\int_{\Omega}u^2v^2+|\nabla v|_{2}^2\right)-\beta^2\left(\int_{\Omega}u^2v^2\right)^2\right]\\
        &\ge4\|u\|_{\lambda}^2\cdot|\nabla v|_{2}^2>0.
    \end{align*}
Moreover, since $|u|_2\ge\delta$ and $|v|_2\le c$ (uniformly in $\beta$), there exists $C>0$ independent of $\beta$ such that
\begin{equation*}\label{bound condition}
    \frac{3|v|_2^2}{|u|_2^2}\le C,\ \ \text{$\forall$ $(u,v)\in\tildeNb\cap\{\mathcal J_\beta\le\kappa\}.$}
\end{equation*}
Thus,
\begin{align*}
3|v|_2^2+\beta|u|_2^2=|u|_2^2\left(\frac{3|v|_2^2}{|u|_2^2}+\beta\right)\le|u|_2^2\left(C+\beta\right)<0,
\end{align*}
if $|\beta|$ is large enough.
Hence, $\det(\mathcal T(u,v))\ne0$, which concludes the proof.
\end{proof}

\begin{proof}[Proof of Theorem \ref{main result 4}]

\noindent \textbf{Step 1.} If $N\le3$, then $m_\beta\le\kappa$, with $\kappa$  independent of $\beta$. Indeed, let $(u,v)\in H$ be such that $u,v\not\equiv0$ and $\text{supp}(u)\cap\text{supp}(v)=\emptyset$ and $\int_\Omega v^3=0$. We consider $(tu,sv)$ with 
    \begin{align*}
        &t^2=\frac{\|u\|_\lambda^2}{|u|_4^4}, &s^2=\frac{|\nabla v|_2^2}{|v|_4^4}.
    \end{align*}
    Hence, $\int_\Omega u^2v^2=0$, $(tu,sv)\in\mathcal N_\beta$ and
    \[
    m_\beta\le\mathcal J_\beta(tu,sv)=\frac14\left[\frac{\|u\|_\lambda^4}{|u|_4^4}+\frac{|\nabla v|_2^2}{|v|_4^4}\right].
    \]

\noindent \textbf{Step 2.}    Let $\{(u_n,v_n)\}$ be a minimizing sequence of $m_\beta$, then by Step 1 we have that
    \[
    \mathcal J_\beta(u_n,v_n)\le m_\beta+o(1)\le\kappa+o(1),
    \]
    where $\kappa$ is independent of $\beta$. The coercivity of $\mathcal J_\beta$ on $\mathcal N_\beta$, given by Lemma \ref{coercività sezione 5}, implies that $\{(u_n,v_n)\}$ is bounded in $H$.
  By Proposition \ref{punti regolari} and Ekeland's Variational Principle \cite{ekeland1974variational}, there exists $\{{\boldsymbol \Lambda}_n\}\subset\R^3$ such that
    \[\mathcal J_\beta'(u_n,v_n)-{\boldsymbol \Lambda}_n\cdot\mathcal G'(u_n,v_n)=o(1),\]
    where $\mathcal G=(\mathcal G_1,\mathcal G_2,\mathcal G_3)$ is defined in the proof of Proposition \ref{punti regolari}.
    According to Lemma \ref{stime norme due}, we have that there exists $\delta>0$ such that $|u_n|_2\ge\delta$, with $\delta>0$ uniform on $n$ and on $\beta$. Reasoning as in Proposition \ref{punti regolari}, we obtain 
    \[\det\mathcal T(u_n,v_n)\le4|u_n|_2^2(C+\beta)\|u_n\|_\lambda^2\cdot|\nabla v_n|_2^2,\]
for some constant $C$ independent of $\beta$ and  $n$. Furthermore,
  by  \eqref{stima basso} and Lemma \ref{stime norme due}, we obtain $\det\mathcal T(u_n,v_n)\le M(C+\beta)$ for a positive constant $ M$ independent of $\beta$ and $n$. Hence, for $|\beta|$ large enough, the matrix $\mathcal T$ is negative definite and there exists a uniform constant $c>0$ such that 
    \[
    \mathcal T(u_n,v_n)\boldsymbol{\Lambda}_n\cdot\boldsymbol{\Lambda}_n\le-c|\boldsymbol{\Lambda}_n|^2.
    \]
    Thus, we can conclude ${\boldsymbol{\Lambda}}_n=\boldsymbol {o(1)}$ as in Lemma \ref{Lemma 1.14}, and $\mathcal J_\beta'(u_n,v_n)=o(1)$.
    Finally, since $\{(u_n,v_n)\}$ is bounded, we can assume that $u_n\wto u$,  $v_n\wto v$ weakly in $H^1(\Omega)$ and the convergence is strong in $L^4(\Omega)$. Hence, 
    \[ o(1)=\langle\mathcal J_\beta'(u_n,v_n)-\mathcal J_\beta'(u,v),(u_n-u,v_n-v)\rangle=\|u_n\|_\lambda^2+|\nabla v_n|_2^2+o(1),\]
    which implies that $(u_n,v_n)\to(u,v)$ strongly in $H$ and that $(u,v)\in\mathcal N_\beta$ by Proposition \ref{chiusura sezione 5}. Since $(u,v)$ is a fully nontrivial critical point for $\mathcal J_\beta$ at level $m_\beta$, we have $m_\beta \geq l_\beta$ and actually we have equality (as $\mathcal{N}_\beta$ contains all fully nontrivial solutions).
        
Moreover, the constant solutions in \eqref{costanti non banali} are not well defined for $\lambda_1>0,\lambda_2=0,\beta<-1$.
\end{proof}
\section{Competitive and indefinite case for \texorpdfstring{$\lambda_1=\lambda_2\le0$}{lambda1=lambda2<=0}}\label{competitive new2}
In this section we assume that $\lambda:=\lambda_1=\lambda_2\le0$ and $\beta<0$, proving Theorems \ref{thm 5} and \ref{thm 6}.  We start with some preliminaries. \\
Let $X_\beta:=\{ (u,v) \in H : u \neq 0, v \neq 0, \mathcal J_\beta' (u,v)= 0 \}$ and $l_\beta=\inf_{X_\beta}\mathcal J_\beta$. The set $X_\beta$ is not empty  for $\beta\in(-1,0)$; indeed,  $(1+\beta)^{-1/2}(\omega,\omega)$ 
are solutions of \eqref{Pb},
where 
\[-\Delta\omega+\lambda\omega=\omega^3\ \ \text{in $H^1(\Omega)$}.\]
On the other hand, we will show  in subsection \ref{competitive NR} that $X_\beta\ne\emptyset$ also for $|\beta|$ large enough following \cite{noris2010existence, xu2022infinitely}. We exploit the assumption $\lambda_1=\lambda_2$ (i.e. the symmetry of the system) to construct a solution of \eqref{Pb}.

\begin{proposition}\label{closure Xb}
 Take $N\geq 4$ and $\beta<0$.     
The set $X_\beta$ is closed if one of the following conditions is satisfied: \begin{itemize}
    \item[(i)] \( |\beta| \) is sufficiently small, and either \( |\lambda| \) is not an eigenvalue or \( \lambda = 0 =\mu_1^N(\Omega) \), 
    \item[(ii)] \( |\beta| \) is sufficiently large and \( \lambda = 0 \).
\end{itemize}
\end{proposition}
\begin{proof}
 \textbf{Step 1}. Assume $u,v\equiv0$.\\ The first part of this step (up to \eqref{stima vett} below) follows the proof of \cite[Lemma 3.3]{clapp2022solutions}.
We consider the decomposition $H=H^+\oplus\tilde H=H^+\oplus H^-\oplus H^0$, with $l:=\dim H_i^-, k:=\dim H_i^0$ for $i=1,2$, as in section \ref{cooperative} and we write 
\[
u_n=u_n^++\tilde u_n=u_n^++u_n^-+u_n^0,\qquad v_n=v_n^++\tilde v_n=v_n^++v_n^-+v_n^0.
\] 
Since $\langle\mathcal J_\beta'(u_n,v_n),(u_n^+,0)\rangle=0$, we obtain that 
    \begin{equation}\label{stima u_n^+}
\begin{aligned}
   C_0 \|u_n^+\|^2\le B_1(u_n^+,u_n^+)&=\int_\Omega u_n^3u_n^++\beta\int_\Omega u_nu_n^+v_n^2 \le|u_n|_4^3\cdot|u_n^+|_4+|\beta|\cdot|u_nu_n^+|_2\cdot|v_n|_4^2\\
   &\le c\|u_n\|^3\cdot\|u_n^+\|+c|\beta|\cdot\|u_n\|\cdot\|u_n^+\|\cdot\|v_n\|^2\le \tilde c\|(u_n,v_n)\|^4,
\end{aligned}
\end{equation}
where $C_0=\mu_{l+k+1}+\lambda$>0.
Similarly, we get $\|u_n^-\|^2\le \tilde c_1\|(u_n,v_n)\|^4$ and also $\|v_n^+\|,\|v_n^-\|\le \tilde c_2\|(u_n,v_n)\|^4$. If $H^0=\{0\}$, i.e. $|\lambda|$ is not an eigenvalue, we conclude that 
\begin{equation}\label{stima vett}
    \|(u_n,v_n)\|^2\ge C_1,\ \text{ for some } C_1>0.
\end{equation}
Then \eqref{stima vett} is a contradiction with the assumption $u_n,v_n\to0$. 

Now, suppose that $H^0\ne\{0\}$, so that  $|\lambda|$  is an eigenvalue. We define $\rho_n:=\|(u_n,v_n)\|$ and $z_n:=\left(u_n/\rho_n,v_n/\rho_n\right)$.
Hence, $\|z_n\|=1$. By \eqref{stima u_n^+}, $\|z_n^+\|=\|u_n^+\|/\rho_n\le c\rho_n=o(1)$.
Thus $z_n^+\to0$, and we have a similar estimates for $\|z_n^-\|$.
Therefore, $z_n^0\to z^0\not\equiv0$. Since $z^0\in H^0$, there exist $\alpha,\gamma\in\R^k$ with 
\[
z^0=\left(\sum_{j=l+1}^{l+k}\alpha_j\varphi_j,\sum_{j=l+1}^{l+k}\gamma_j\varphi_j\right),
\]
where  $|\alpha|^2+|\gamma|^2=1$. Testing the equations in \eqref{Pb} with $z^0$, we obtain 
\begin{align*}
    0=\int_\Omega u_n^3z_1^0+\beta\int_\Omega u_nv_n^2z_1^0,\qquad 0=\int_\Omega v_n^3z_2^0+\beta\int_\Omega v_nu_n^2z_2^0.
\end{align*} Dividing by $\rho_n^3$, we obtain 
\begin{equation}\label{eq star}
\begin{aligned}
    0=\int_\Omega |z_1^0|^4+\beta\int_\Omega|z_1^0|^2|z_2^0|^2,\qquad 0=\int_\Omega|z_2^0|^4+\beta\int_\Omega|z_1^0|^2|z_2^0|^2.
\end{aligned}
\end{equation}
If either $z_1^0\equiv0$ or $z_2^0\equiv0$, we obtain $z^0\equiv(0,0)$, a contradiction. Hence,  $z_1^0,z_2^0\not\equiv0$. By H\"older's inequality, we infer that 
\begin{align*}\left(\int_\Omega |z_1^0|^4\right)\left(\int_\Omega|z_2^0|^4\right)&=\beta^2\left(\int_\Omega|z_1^0|^2|z_2^0|^2\right)^2\le\beta^2\left(\int_\Omega |z_1^0|^4\right)\left(\int_\Omega|z_2^0|^4\right),\end{align*} 
which is a contradiction for $\beta\in(-1,0)$. 

Let us now consider the case $|\beta|$ large.  
Since $\|z^0\|=1$, we have (up to a subsequence) that $z^0=:z_\beta^0\to z_\infty^0$ as $\beta\to-\infty$ and $\|z_\infty^0\|=1$. Now, if 
\[\lim_{\beta\to-\infty}\beta\int_\Omega|z_{\beta,1}^0|^2|z_{\beta,2}^0|^2=-\infty,\] we get a contradiction due to \eqref{eq star} . Otherwise,
\begin{equation}\label{eq star2}
\int_\Omega|z_{\infty,1}^0|^2|z_{\infty,2}^0|^2=\lim_{\beta\to-\infty}\int_\Omega|z_{\beta,1}^0|^2|z_{\beta,2}^0|^2=0.
\end{equation}
Therefore, since the eigenfunctions can not be zero in a open subset $\tilde\Omega\subset\Omega$ by the Unique Continuation Principle, we can assume that $z_{\infty,1}^0\equiv0$. 
Exploiting equations \eqref{eq star2} and \eqref{eq star}, we obtain the contradiction $z_\infty^0\equiv(0,0)$.\\
   \textbf{Step 2:}
   Assume that either $u\equiv0$ or $v\equiv0$. We focus on the case $u\equiv0$. In particular, there exists $\delta_\beta>0$ such that $\|v_n\|\ge\delta_\beta>0$. We define $z_n=u_n/\|u_n\|$ and $z$ the weak limit of $z_n$ in $H^1(\Omega)$. Assume that $z\equiv0$. Testing \eqref{Pb} with $(u_n,0)$ and dividing by $\|u_n\|^2$ we obtain:
    \[\int_\Omega|\nabla z_n|^2+\lambda\int_\Omega z_n^2\le\|u_n\|^{-2}\int_\Omega u_n^4\le C\|u_n\|^2.\]
    Thus $z_n\to0$, and this is not possible because $\|z_n\|=1$. Now, assume that $\lambda=0$.
    Since $\langle\mathcal J_\beta'(u_n,v_n),(1,0)\rangle=0$, we obtain
    \[
    0=\int_\Omega z_n^3+\frac{\beta}{\|u_n\|^2}\int_\Omega z_nv_n^2,
    \]
    and we get a contradiction passing to the limit as $n\to+\infty$.  This, also recalling Step 1, shows $(i)-(ii)$ in the case $\lambda=0$.
   Finally, assume that $|\lambda|$ is not an eigenvalue, this case follows by \cite[ proof of Theorem 1.3]{xu2025least}. 
\end{proof}

\subsection{Weakly competitive for \texorpdfstring{$\lambda:=\lambda_1=\lambda_2\le0$}{lambda1=lambda2<=0}}
Here, we show Theorem \ref{thm 5}.
We start proving compactness properties of the energy functional $\mathcal J_\beta$.

\begin{theorem}\label{compattezza w.c.}
    Let $\beta\in(-1,0)$ and let $N\le 4$. The following results hold
    \begin{itemize}
        \item[$(i)$] if $N\le3$, then the $(PS)_c$-condition is satisfied for every $c\in\R$,
        \item[$(ii)$] if $N=4$, then the $(PS)_c$-condition is satisfied for $c<L_\lambda+\frac{S^2}{8}$.
    \end{itemize}
\end{theorem}
\begin{proof}
    Let $\{(u_n,v_n)\}$ be a Palais-Smale sequence at level $c\in\R$. The item $(i)$ can be proved as in \cite[Lemma 3.1]{xu2025least}. Now, we focus on $(ii)$. Since $\beta\in(-1,0)$, we have:
    \[
    c+o(1)=\mathcal J_\beta(u_n,v_n)=\frac14\int_\Omega(u_n^4+2\beta u_n^2v_n^2+v_n^4)+o(1) \ge(\underbrace{1-|\beta|}_{>0})\int_\Omega (u_n^4+v_n^4)+o(1),
    \]
    which implies that there exists $C_1>0$ such that $|u_n|_4,|v_n|_4\le C_1$ and also $|u_n|_2,|v_n|_2\le C_1$. Moreover, 
    \[\frac14B_1(u_n,u_n)+\frac 14B_2(v_n,v_n)+o(1)=\frac14\int_\Omega (u_n^4+2\beta u_n^2v_n^2+v_n^4).\]
  Then
 $\|(u_n,v_n)\|\le C_2$, for some $C_2>0$. The rest of the proof is the same of Theorem \ref{Compactness beta<0}.
\end{proof}

\begin{proposition}\label{soglia w.c.}
    Let $N=4$, $\lambda\in\R$ and $|\beta|$  small enough, then $l_\beta=\inf_{X_\beta}\mathcal J_\beta<L_{\lambda}+\frac{S^2}{8}$.
\end{proposition}
\begin{proof}
Testing the energy level $l_\beta$, see \eqref{db}, with
    \[( u, v)=\left(\frac{1}{\sqrt{1+\beta}}\omega,\frac{1}{\sqrt{1+\beta}}\omega\right),\]
    we obtain
    \[l_\beta\le\frac{L_\lambda+L_\lambda}{1+\beta}=\frac{2L_\lambda}{1+\beta}.\]
       Thus $\lim_{\beta\to0^-}l_\beta\le 2L_\lambda<L_{\lambda}+\frac{S^2}{8}.$ Hence,
        we can find $\beta_1>0$ such that, for every $|\beta|<\beta_1$, $l_\beta<L_\lambda+\frac{S^2}{8}$.
\end{proof}
\begin{proof}[Proof of Theorem \ref{thm 5}]
Let $\{(u_n,v_n)\}\subset X_\beta$ be a minimizing sequence for $l_\beta$ (defined in \eqref{db}). Then $\{(u_n,v_n)\}$ is a Palais-Smale sequence at level $l_\beta$. By Proposition \ref{soglia w.c.} and Theorem \ref{compattezza w.c.}, we infer that $u_n\to u, v_n\to u$ strongly in $H^1(\Omega)$, up to a subsequence. 
 Moreover, since $\lambda_1=\lambda_2\le0$ and $\beta\in(-1,0)$ the fully non-trivial constant solutions in \eqref{costanti non banali} are not well defined, and the proof is complete.
\end{proof}

\subsection{Strongly  competitive case for \texorpdfstring{$\lambda_1=\lambda_2\le0$}{lambda1=lambda2<=0} and \texorpdfstring{$N\le3$}{N<=3}}\label{competitive NR}
Here, we prove Theorem \ref{thm 6}.
We adapt \cite{noris2010existence, xu2022infinitely} to the Neumann case.
We define the energy functional $J:H\to\R$,
\[J(u,v)=\frac12\|u\|^2+\frac12\|v\|^2+\frac\mu 2\int_\Omega(u_+)^2+\frac\mu 2\int_\Omega(v_+)^2-\frac14\int_\Omega(|u_+|^4+|v_+|^4+2\beta u^2v^2).\]
where $\mu:=\lambda-1$ and $u_+,v_+$ are the positive parts of $u,v$ respectively.
We use the minimax principle developed by Noris and Ramos in \cite{noris2010existence}.\\ We define $H_{k}=\text{span}\{\varphi_1,\dots,\varphi_{k}\}$ the subspace of dimension $k$ spanned by the first $k-$eigenfunctions of $-\Delta+\text{Id}$ with Neumann boundary conditions.
We define 
\[S_k:=\left\{(u,v)\in H\ :\ u-v\in H_{k}^{\perp},\ \|u-v\|=\rho_k\right\},\]
where $\rho_k>0$,  and we define $Q_k:=B_{R_k}\cap H_k$, where $R_k>\rho_k$. Finally, we set the  minimax class:
\[\Gamma_k:=\left\{\gamma\in C(Q_k;H)\ :\ \gamma(-u)=(\gamma_2(u),\gamma_1(u))\  in \ Q_k,\ \gamma(u)=(u_+,u_-)\ on\ \partial Q_k\right\},\]
\begin{equation}\label{ck}d_k:=\inf_{\gamma\in\Gamma_k}\max_{u\in Q_k}J(\gamma(u)).\end{equation}
It is standard to prove that  $(\Gamma_k,d_k)$ is an admissible minimax class and $d_k$ is a critical level when $J$ satisfies the $(PS)_{d_k}$-condition, see \cite[ Proposition 2.4]{noris2010existence}.  
Taking $\gamma(z)=(z_+,z_-)$, we deduce that 
\[ J(\gamma(z))=J(z_+,z_-)\le\frac{c_\mu}{2}\|z\|^2-\frac14\int_\Omega z^4<0,\ \ \forall\ z\in \partial Q_k,\]
with $R_k>0$ large enough and $c_\mu>0$ is a positive constant such that $c_\mu=1$ when $\mu\le0$.
\begin{lemma}[{\cite[Lemma 2.1]{noris2010existence} and \cite[Lemma 3.2]{xu2022infinitely}}]\label{geometry NR}
    For any $\alpha>0$ there exists $\rho_k$, $k_0\ge1$ and $\beta_0>1$ such that 
    \[
    \inf_{(u,v)\in S_k}J(u,v)\ge\alpha,\ \ \forall\ \ k\ge k_0,\quad |\beta|>\beta_0.
    \]
\end{lemma}

As a consequence, for each $\alpha>0$, $d_k\ge\alpha$ for $k$ large enough and $d_k\to+\infty$, see also \cite[Lemma 2.2]{noris2010existence}.
\begin{proposition}[{\cite[Lemma 3.5]{xu2022infinitely}}]\label{PS condition NR}
    Assume that $\beta\ne-1$ and $c>0$ Then
    \begin{itemize}
        \item[$(i)$] any Palais-Smale sequence of $J$ at level $c$ is bounded;
        \item[$(ii)$] the functional $J$ satisfies the $(PS)_c$ condition.
    \end{itemize}
    The same holds also for the functional $\mathcal J_\beta$.
\end{proposition}

\begin{proof}[Proof of Theorem \ref{thm 6}]
Fix $\alpha>0$. Let $d_k$ be the critical level defined in \eqref{ck} and choose $k\ge1$ large enough such that $d_k\ge\alpha>0$ and $d_k\to+\infty$ by Lemma \ref{geometry NR}. Thus,  Proposition \ref{PS condition NR} implies the existence of a critical point $(u_k,v_k)$ of $J$ at level $d_k$. Moreover, $(u_k,v_k)$ is not constant, otherwise $u_k^2=v_k^2\equiv\frac{\lambda}{1+\beta}$ by \eqref{costanti non banali} and 
   \[
   0<\alpha\le J(u_k,v_k)=\frac{\lambda^2}{2(1+\beta)}\le0,
   \]
   which is a contradiction. We recall also that any critical point of $J$ is a weak solution of \eqref{Pb}, see \cite[Lemma 4.1]{dancer2010priori}.
   Hence, $(i)$ is proved  
    and we have also that
 \[  X_\beta=\{ (u,v) \in H : u \neq 0, v \neq 0, \mathcal J_\beta' (u,v)= 0 \}\ne\emptyset.\]
    Let $\{(u_n,v_n)\}\subset X_\beta$ 
    be such that $\mathcal J_\beta(u_n,v_n)=l_\beta+o(1)$. Proposition \ref{PS condition NR} implies also that $\mathcal J_\beta$ satisfies the $(PS)_{l_\beta}$-condition and that there exists $(u,v)\in H$ such that $u_n\to u,v_n\to v$ in $H^1(\Omega)$. According to Proposition \ref{closure Xb}-$(ii)$,  $u,v\not\equiv0$ for $\lambda=0$, thus $(ii)$ holds.
    \end{proof}

\begin{remark}
When $\lambda<0$ and $\beta<-1$, one has
\[
l_\beta \;\le\; \mathcal J_\beta(c_1,c_2) \;=\; \frac{\lambda^2}{2(1+\beta)} \;<\;0\;<\;L_\lambda,
\]
so in particular $l_\beta$ is attained. However, in this situation we cannot conclude whether the least energy solution is constant or not. Moreover, this argument works only under Neumann boundary conditions.
\end{remark}

\section*{Acknowledgments} 
Delia Schiera and Hugo Tavares are partially supported by the Portuguese government through FCT- Funda\c c\~ao para a Ci\^encia e a Tecnologia, I.P., under the projects UIDB/04459/2020 with DOI identifier 10-54499/UIDP/04459/2020 (CAMGSD).

Delia Schiera is also partially supported by
FCT, I.P., under the Scientific Employment Stimulus - Individual
Call (CEEC Individual), with DOI identifier 10.54499/2020.02540.CEECIND/CP1587/CT0008, and by GNAMPA (Gruppo Nazionale per l'Analisi, Probabilit\`a e le loro Applicazioni) -- INdAM (Istituto Nazionale di Alta Matematica), through the project `Critical and limiting phenomena in nonlinear elliptic systems', CUP E5324001950001. 

The author Simone Mauro was supported by a PhD scholarship (granted by the University of Calabria) and thanks Instituto Superior Técnico, Universidade de Lisboa, for the hospitality during the visiting period in which this work was developed.

\noindent\textbf{Simone Mauro}\\
 Dipartimento di Matematica e Informatica, \\
 Universit\`a della Calabria\\
 Ponte Pietro Bucci cubo 31B\\
 87036 Arcavacata di Rende, Cosenza, Italy\\
\texttt{simone.mauro@unical.it}

\bigskip

\noindent\textbf{Delia Schiera and Hugo Tavares}\\
CAMGSD - Centro de An\'alise Matem\'atica, Geometria e Sistemas Din\^amicos\\
Departamento de Matem\'atica do Instituto Superior T\'ecnico\\
Universidade de Lisboa\\
1049-001 Lisboa, Portugal\\
\texttt{hugo.n.tavares@tecnico.ulisboa.pt}


\begin{thebibliography}{10}

\bibitem{adimurthimancini1}
Adimurthi and Giovanni Mancini.
\newblock The {N}eumann problem for elliptic equations with critical
  nonlinearity.
\newblock In {\em Nonlinear analysis}, Sc. Norm. Super. di Pisa Quaderni, pages
  9--25. Scuola Norm. Sup., Pisa, 1991.

\bibitem{adimurthi1990critical}
Adimurthi and S.L. Yadava.
\newblock Critical {S}obolev exponent problem in $\mathbb{R}^n$ $(n\ge4)$ with
  {N}eumann boundary condition.
\newblock {\em Proceedings-Mathematical Sciences}, 100(3):275--284, 1990.

\bibitem{ambrosetti2007standing}
Antonio Ambrosetti and Eduardo Colorado.
\newblock Standing waves of some coupled nonlinear schr{\"o}dinger equations.
\newblock {\em Journal of the London Mathematical Society}, 75(1), 2007.

\bibitem{ambrosetti1986note}
Antonio Ambrosetti and Michael Struwe.
\newblock A note on the problem $-\triangle u=\lambda u+|u|^{2^*-2}u$.
\newblock {\em Manuscripta Math.}, 54(4):373--379, 1986.

\bibitem{breziskato1}
Ha\"im Br\'ezis and Tosio Kato.
\newblock Remarks on the {S}chr\"odinger operator with singular complex
  potentials.
\newblock {\em J. Math. Pures Appl. (9)}, 58(2):137--151, 1979.

\bibitem{brezis1983relation}
Ha{\"\i}m Br{\'e}zis and Elliott Lieb.
\newblock A relation between pointwise convergence of functions and convergence
  of functionals.
\newblock {\em Proceedings of the American Mathematical Society},
  88(3):486--490, 1983.

\bibitem{brezisnirenberg}
Ha\"im Br\'ezis and Louis Nirenberg.
\newblock Positive solutions of nonlinear elliptic equations involving critical
  {S}obolev exponents.
\newblock {\em Comm. Pure Appl. Math.}, 36(4):437--477, 1983.

\bibitem{byeon2017pattern}
Jaeyoung Byeon, Yohei Sato, and Zhi-Qiang Wang.
\newblock Pattern formation via mixed interactions for coupled schr{\"o}dinger
  equations under neumann boundary condition.
\newblock {\em Journal of Fixed Point Theory and Applications}, 19:559--583,
  2017.

\bibitem{capozzi1985existence}
Alberto Capozzi, Donato Fortunato, and Giuliana Palmieri.
\newblock An existence result for nonlinear elliptic problems involving
  critical sobolev exponent.
\newblock In {\em Annales de l'Institut Henri Poincar{\'e} C, Analyse non
  lin{\'e}aire}, volume~2, pages 463--470. Elsevier, 1985.

\bibitem{chabrowski2001neumann}
J~Chabrowski and Jianfu Yang.
\newblock On the neumann problem for an elliptic system of equations involving
  the critical sobolev exponent.
\newblock In {\em Colloquium Mathematicae}, volume~90, pages 19--35, 2001.

\bibitem{chen2012positive}
Zhijie Chen and Wenming Zou.
\newblock Positive least energy solutions and phase separation for coupled
  schr{\"o}dinger equations with critical exponent.
\newblock {\em Archive for Rational Mechanics and Analysis}, 205(2):515--551,
  2012.

\bibitem{cherrier1984meilleures}
Pascal Cherrier.
\newblock Meilleures constantes dans des in{\'e}galit{\'e}s relatives aux
  espaces de sobolev.
\newblock {\em Bull. Sci. Math}, 108(2):225--262, 1984.

\bibitem{clapp2022solutions}
M\'onica Clapp and Andrzej Szulkin.
\newblock Solutions to indefinite weakly coupled cooperative elliptic systems.
\newblock {\em Topol. Methods Nonlinear Anal.}, 59(2A):553--568, 2022.

\bibitem{comte1990solutions}
Myriam Comte and Mariette~C. Knaap.
\newblock Solutions of elliptic equations involving critical sobolev exponents
  with neumann boundary conditions.
\newblock {\em {M}anuscripta {M}athematica}, 69(1):43--70, 1990.

\bibitem{comte1991existence}
Myriam Comte and Mariette~C. Knaap.
\newblock Existence of solutions of elliptic equations involving critical
  sobolev exponents with neumann boundary condition in general domains.
\newblock {\em Differential Integral Equations}, 4(6):1133--1146, 1991.

\bibitem{dancer2010priori}
EN~Dancer, Juncheng Wei, and Tobias Weth.
\newblock A priori bounds versus multiple existence of positive solutions for a
  nonlinear schr{\"o}dinger system.
\newblock In {\em Annales de l'IHP Analyse non lin{\'e}aire}, volume~27, pages
  953--969, 2010.

\bibitem{ekeland1974variational}
Ivar Ekeland.
\newblock On the variational principle.
\newblock {\em Journal of Mathematical Analysis and Applications},
  47(2):324--353, 1974.

\bibitem{esry1997hartree}
BD~Esry, CH~Greene, J~Burke~Jr, and JL~Bohn.
\newblock Hartree-fock theory for double condensates.
\newblock {\em Physical Review Letters}, 78(19):3594--3597, 1997.

\bibitem{kou2022existence}
Bingyu Kou and Tianqing An.
\newblock The existence of positive solutions for the neumann problem of
  p-laplacian elliptic systems with sobolev critical exponent.
\newblock {\em Boundary Value Problems}, 2022(1):22, 2022.

\bibitem{lin1988large}
C-S Lin, W-M Ni, and Izumi Takagi.
\newblock Large amplitude stationary solutions to a chemotaxis system.
\newblock {\em Journal of Differential Equations}, 72(1):1--27, 1988.

\bibitem{lin2006diffusion}
Chang-Shou Lin and Wei-Ming Ni.
\newblock On the diffusion coefficient of a semilinear neumann problem.
\newblock In {\em Calculus of Variations and Partial Differential Equations:
  Proceedings of a Conference held in Trento, Italy June 16--21, 1986}, pages
  160--174. Springer, 2006.

\bibitem{liu2013neumann}
Zhaoxia Liu.
\newblock On a neumann problem with critical exponents and hardy potentials.
\newblock {\em Journal of Mathematical Analysis and Applications},
  399(1):50--74, 2013.

\bibitem{maia2006positive}
Liliane~A Maia, Eugenio Montefusco, and Benedetta Pellacci.
\newblock Positive solutions for a weakly coupled nonlinear schr{\"o}dinger
  system.
\newblock {\em Journal of Differential Equations}, 229(2):743--767, 2006.

\bibitem{malomed2005multidimensional}
Boris~A Malomed, Dumitru Mihalache, Frank Wise, and Lluis Torner.
\newblock Multidimensional solitons: Well-established results and novel
  findings.
\newblock {\em Journal of Optics B: Quantum and Semiclassical Optics},
  7(5):R53, 2005.

\bibitem{mandel2015minimal}
Rainer Mandel.
\newblock Minimal energy solutions for cooperative nonlinear schr{\"o}dinger
  systems.
\newblock {\em NoDEA Nonlinear Differential Equations Appl}, 22(2):239--262,
  2015.

\bibitem{merle2003stability}
Frank Merle and Pierre Rapha{\"e}l.
\newblock Stability of standing waves for nonlinear schr{\"o}dinger equations.
\newblock {\em Bulletin of the American Mathematical Society}, 41(3):349--354,
  2003.

\bibitem{moser1960new}
J\"urgen Moser.
\newblock A new proof of de giorgi's theorem concerning the regularity problem
  for elliptic differential equations.
\newblock {\em Communications on Pure and Applied Mathematics}, 13(3):457--468,
  1960.

\bibitem{nardi2015schauder}
Giacomo Nardi.
\newblock Schauder estimate for solutions of poisson's equation with neumann
  boundary condition.
\newblock {\em L'enseignement Math{\'e}matique}, 60(3):421--435, 2015.

\bibitem{noris2010existence}
Benedetta Noris and Miguel Ramos.
\newblock Existence and bounds of positive solutions for a nonlinear
  schr{\"o}dinger system.
\newblock {\em Proceedings of the American Mathematical Society},
  138(5):1681--1692, 2010.

\bibitem{pistoia2017spiked}
Angela Pistoia and Hugo Tavares.
\newblock Spiked solutions for schr{\"o}dinger systems with sobolev critical
  exponent: the cases of competitive and weakly cooperative interactions.
\newblock {\em Journal of Fixed Point Theory and Applications}, 19(1):407--446,
  2017.

\bibitem{rabinowitz78linking}
Paul~H. Rabinowitz.
\newblock Some critical point theorems and applications to semilinear elliptic
  partial differential equations.
\newblock {\em Ann. Scuola Norm. Sup. Pisa Cl. Sci. (4)}, 5(1):215--223, 1978.

\bibitem{saldana2022least}
Alberto Salda\~na and Hugo Tavares.
\newblock On the least-energy solutions of the pure neumann lane--emden
  equation.
\newblock {\em Nonlinear Differential Equations and Applications NoDEA},
  29(3):30, 2022.

\bibitem{sirakov2007least}
Boyan Sirakov.
\newblock Least energy solitary waves for a system of nonlinear schr{\"o}dinger
  equations in.
\newblock {\em Communications in mathematical physics}, 271(1):199--221, 2007.

\bibitem{sun2025brezis}
Liming Sun, Juncheng Wei, and Wen Yang.
\newblock On brezis' first open problem: A complete solution.
\newblock {\em arXiv preprint arXiv:2503.06904}, 2025.

\bibitem{szulkin2010method}
Andrzej Szulkin and Tobias Weth.
\newblock The method of nehari manifold.
\newblock {\em Handbook of nonconvex analysis and applications}, 597632, 2010.

\bibitem{SzulkinWethWillem}
Andrzej Szulkin, Tobias Weth, and Michel Willem.
\newblock Ground state solutions for a semilinear problem with critical
  exponent.
\newblock {\em Differential Integral Equations}, 22(9-10):913--926, 2009.

\bibitem{tavares2012existence}
H.~Tavares and T.~Weth.
\newblock Existence and symmetry results for competing variational systems.
\newblock {\em NoDEA Nonlinear Differential Equations Appl.}, 20(3):715--740,
  2013.

\bibitem{wang1991neumann}
Xu-Jia Wang.
\newblock Neumann problems of semilinear elliptic equations involving critical
  sobolev exponents.
\newblock {\em Journal of Differential Equations}, 93(2):283--310, 1991.

\bibitem{wei2008existence}
Gong-Ming Wei and Yan-Hua Wang.
\newblock Existence of least energy solutions to coupled elliptic systems with
  critical nonlinearities.
\newblock {\em Electron. J. Differential Equations}, pages No. 49, 8, 2008.

\bibitem{weiyao2012uniqueness}
Juncheng Wei and Wei Yao.
\newblock Uniqueness of positive solutions to some coupled nonlinear
  {S}chr\"odinger equations.
\newblock {\em Commun. Pure Appl. Anal.}, 11(3):1003--1011, 2012.

\bibitem{willem2012minimax}
Michel Willem.
\newblock {\em Minimax theorems}, volume~24.
\newblock Springer Science \& Business Media, 2012.

\bibitem{xu2025least}
Ruijin Xu, Jiabao Su, and Rushun Tian.
\newblock Least energy solutions of some indefinite schr{\"o}dinger systems.
\newblock {\em Journal of Differential Equations}, 424:501--525, 2025.

\bibitem{xu2022infinitely}
Ruijin Xu and Rushun Tian.
\newblock Infinitely many vector solutions of a fractional nonlinear
  schr{\"o}dinger system with strong competition.
\newblock {\em Applied Mathematics Letters}, 132:108187, 2022.

\bibitem{zhang2012multiple}
Yajing Zhang.
\newblock Multiple solutions of an inhomogeneous neumann problem for an
  elliptic system with critical sobolev exponent.
\newblock {\em Nonlinear Analysis: Theory, Methods \& Applications},
  75(4):2047--2059, 2012.

\end{thebibliography}
\end{document}